\documentclass[12pt]{amsart}
\usepackage{amsmath,amsthm,amssymb,amsfonts,enumerate,color}
\usepackage{amsmath, amsthm, amsfonts, amssymb, color}
\usepackage{mathrsfs}
\usepackage{ amsmath, amsthm, amsfonts, amssymb, color}
 \usepackage{mathrsfs}
\usepackage{amsfonts, amsmath}
 \usepackage{amsmath,amstext,amsthm,amssymb,amsxtra}
 \usepackage{txfonts} 
 \usepackage[colorlinks, citecolor=blue,pagebackref,hypertexnames=false]{hyperref}
 \allowdisplaybreaks
 \usepackage{pgf,tikz}
\begin{document}




\newcommand{\norm}[1]{\left\Vert#1\right\Vert}
\newcommand{\abs}[1]{\left\vert#1\right\vert}
\newcommand{\set}[1]{\left\{#1\right\}}
\newcommand{\Real}{\mathbb{R}}
\newcommand{\RR}{\mathbb{R}^n}
\newcommand{\supp}{\operatorname{supp}}
\newcommand{\card}{\operatorname{card}}
\renewcommand{\L}{\mathcal{L}}
\renewcommand{\P}{\mathcal{P}}
\newcommand{\T}{\mathcal{T}}
\newcommand{\A}{\mathbb{A}}
\newcommand{\K}{\mathcal{K}}
\renewcommand{\S}{\mathcal{S}}
\newcommand{\blue}[1]{\textcolor{blue}{#1}}
\newcommand{\red}[1]{\textcolor{red}{#1}}
\newcommand{\Id}{\operatorname{I}}

\newtheorem{thm}{Theorem}[section]
\newtheorem{prop}[thm]{Proposition}
\newtheorem{cor}[thm]{Corollary}
\newtheorem{lem}[thm]{Lemma}
\newtheorem{lemma}[thm]{Lemma}
\newtheorem{exams}[thm]{Examples}
\theoremstyle{definition}
\newtheorem{defn}[thm]{Definition}
\newtheorem{rem}[thm]{Remark}

\numberwithin{equation}{section}

\title[Poisson integrals of Schr\"odinger operators with BMO traces]
{On characterization of   Poisson integrals of Schr\"odinger operators with BMO traces}

 \author[X.T. Duong, L.X. Yan and C. Zhang]{Xuan Thinh Duong, \ Lixin Yan \  and\ Chao Zhang}

\address {Department of Mathematics, Macquarie University, NSW 2109, Australia}
\email{xuan.duong@mq.edu.au}
 \address{Department of Mathematics\\ Zhongshan University \\
         Guangzhou 510275, PR China}
 \email{mcsylx@mail.sysu.edu.cn}

 \address{School of Statistics and Mathematics \\
             Zhejiang Gongshang University \\
             Hangzhou 310018, PR China}
 \email{zaoyangzhangchao@163.com}

\begin{abstract}  Let $\L$ be a Schr\"odinger operator of the form $\L=-\Delta+V$ acting on $L^2(\mathbb R^n)$ where the nonnegative
potential $V$ belongs to  the reverse H\"older class    $B_q$ for some $q\geq n.$
 Let ${\rm BMO}_{{\mathcal{L}}}(\RR)$ denote the BMO
space on $\RR$ associated to the Schr\"odinger operator $\L$.
In this article we will
show that
 a function  $f\in {\rm BMO}_{{\mathcal{L}}}(\RR)$ is the trace of the solution of
${\mathbb L}u=-u_{tt}+\L u=0,  u(x,0)= f(x),$
 where $u$ satisfies  a Carleson condition
\begin{eqnarray*}
 \sup_{x_B, r_B} r_B^{-n}\int_0^{r_B}\int_{B(x_B, r_B)}  t|\nabla   u(x,t)|^2 {dx dt }  \leq C <\infty.
\end{eqnarray*}
Conversely, this Carleson   condition characterizes  all the ${\mathbb L}$-harmonic functions whose traces belong to
the space ${\rm BMO}_{{\mathcal{L}}}(\RR)$.
This result extends the  analogous characterization founded by Fabes, Johnson and Neri in \cite{FJN}
for   the classical BMO space  of John and Nirenberg.


 \end{abstract}

\maketitle

 \tableofcontents

\section{Introduction and statement of the main result}
\setcounter{equation}{0}

Consider the Laplace operator $\Delta=\sum_{i=1}^n\partial_{x_i}^2$ on the Euclidean space $\mathbb R^n$. A basic tool
in harmonic analysis to study a (suitable) function $f(x)$ on $\mathbb R^n$ is to consider a harmonic function
on $\mathbb R^{n+1}_{+}$ which has the boundary value as $f(x)$. A standard choice for such a harmonic
function is the Poisson
integral $e^{-t\sqrt{-\Delta}} f(x)$ and one recovers $f(x)$ when letting $t \rightarrow 0^{+}$. In other words,
one obtains $u(x,t) = e^{-t\sqrt{-\Delta}} f(x)$ as the solution of the equation  $u_{tt} +\Delta u=0,  u(x,0)= f(x)$.
This approach is intimately related to
the study of singular integrals. For the classical case $f \in L^p(\mathbb R^n)$, $1 \le p \le \infty$, we refer the reader to Chapter 2
of the standard textbook  \cite{SW}.

At the end-point space $L^{ \infty}(\RR)$, the study of singular integrals has a natural substitution, the BMO space, i.e. the space
of functions of bounded mean oscillation.
A celebrated   theorem of  Fefferman and Stein \cite{FS}   states    that
 a BMO function   is the trace of the solution of
 $\partial_{tt}u +\Delta u=0,  u(x,0)= f(x),$
 whenever $u$ satisfies
\begin{eqnarray} \label{e1.1}
 \sup_{x_B, r_B} r_B^{-n}\int_0^{r_B}\int_{B(x_B, r_B)}  t|\nabla  u(x,t)|^2 {dx dt }  \leq C<\infty,
\end{eqnarray}
where $\nabla=(\nabla_x, \partial_t).$  Expanding on this result,
   Fabes, Johnson and Neri \cite{FJN} showed that  condition  \eqref{e1.1}
   characterizes  all the harmonic functions whose traces are in ${\rm BMO}(\RR)$.
  The study of this topic has been widely
extended to more general operators such as elliptic operators (instead of the Laplacian)
and for domains other than $\mathbb R^n$ such as Lipschitz domains. See for examples \cite{DKP, FN1, FN,  HMM}.


The main aim of this article is to study a similar characterization to \eqref{e1.1}  for the Schr\"odinger operators
with appropriate conditions on its potentials.
Let us consider  the Schr\"odinger  operator
\begin{equation}\label{e1.2}
\L  =-\Delta  +V(x) \ \ \ {\rm on} \ \ L^2(\RR), \ \ \ \  n\geq 3.
\end{equation}
As to the nonnegative  potential $V$, we   assume that it is not identically zero and that
 $
V\in B_q$  for  some $q\geq n/2$,
which by definition means that   $V\in L^{q}_{\rm loc}(\RR), V\geq 0$, and
there exists a  constant $C>0$ such that   the reverse H\"older inequality
\begin{equation}\label{e1.3}
\left(\frac{1}{\abs{B}}\int_BV(y)^q~dy\right)^{1/q}\leq\frac{C}{\abs{B}}\int_BV(y)~dy,
\end{equation}
holds for all   balls $B$ in $\RR.$

The operator   $\L$ is a self-adjoint
operator on $L^2(\RR)$. Hence $\L$ generates the $\L$-Poisson semigroup $\{e^{-t\sqrt{\L}}\}_{t>0}$ on $L^2(\RR)$.
Since the potential $V$ is nonnegative,  the semigroup kernels ${\mathcal P}_t(x,y)$  of the operators $e^{-t\sqrt{\L}}$ satisfy
\begin{eqnarray*}
0\leq {\mathcal P}_t(x,y)\leq p_t(x-y)
\end{eqnarray*}
for all $x,y\in\RR$ and $t>0$, where
$$
p_t(x-y)=c_n {t\over (t^2 +|x|^2)^{{n+1\over 2}}}, \ \ \ \ c_n={\Gamma\Big({n+1\over 2}\Big) \over \pi^{(n+1)/2}}
$$
is the kernel of the classical Poisson semigroup
$\set{{  P}_t}_{t>0}=\{e^{-t\sqrt{-\Delta}}\}_{t>0}$ on $\Real^n$.
For $f\in  L^p(\RR)$, $1\leq p<  \infty,$
it is well known that the Poisson extension $u(x,t)=e^{-t\sqrt{\L}}f(x), t>0, x\in\RR$, is a solution to the equation
\begin{eqnarray}\label{el.4}
{\mathbb L}u=-u_{tt}+{\L} u =0\ \ \ {\rm in }\ {\mathbb R}^{n+1}_+
\end{eqnarray}
with the boundary data $f$ on $\RR$
 (see Remark~\ref{re3.3} for $p=\infty$ below).
 The equation  ${\mathbb L}u =0$ is understood in the weak sense, that is,
  $u\in {W}^{1, 2}_{{\rm loc}} ( {\mathbb R}^{n+1}_+) $ is a weak solution of ${\mathbb L}u =0$   if it satisfies
$$\int_{{\mathbb R}^{n+1}_+} {\nabla}u\cdot {\nabla}\psi \,dY+
 \int_{{\mathbb R}^{n+1}_+} V u\psi \,dY=0,\ \ \ \ \forall \psi\in C_0^{1}({\mathbb R}^{n+1}_+).
 $$
 In the sequel,   we call such a function $u$  an ${\mathbb L}$-harmonic function associated to the operator ${\mathbb L}$.

As mentioned above,  we  are interested in  deriving the characterization of the solution
 to the equation $ {\mathbb L}u  =0$ in ${\mathbb R}^{n+1}_+$ having boundary values with BMO data.
Following \cite{DGMTZ},   a locally integrable function $f$ belongs to  BMO$_{{\mathcal{L}}}({\mathbb R}^n)$
whenever there is constant $C\geq 0$ so that
\begin{equation} \label{e1.5}
  {1\over |B|}\int_{B}|f(y)- f_{B}|dy\leq C
  \end{equation}
  for every ball $B=B(x, r)$, and
\begin{equation}\label{e1.6}
      {1\over |B|}\int_{B}|f(y)|dy\leq C
\end{equation}
 for every  ball   $B=B(x,r)$ with  $r\geq\rho(x)$.   Here $f_B=|B|^{-1}\int_B f(x) dx$ and
the critical radii above are determined by the function $\rho(x; V)=\rho(x)$ which takes the explicit form
\begin{equation}\label{e1.7}
 \rho(x)=\sup \Big{\{} r>0: \ {1\over r^{n-2}} \int_{B(x, r)} V(y)dy \leq 1 \Big{\}}.
\end{equation}
We define $\|f\|_{{\rm BMO}_{\mathcal{L}}(\RR)}$ to be the smallest
 $C$  in the right hand sides of \eqref{e1.5} and \eqref{e1.6}.
Because of \eqref{e1.6}, this ${\rm BMO}_{{\mathcal{L}}}(\RR)$ space is in fact a proper subspace of  the classical BMO space of John and Nirenberg,
 and it  turns out to be a suitable space in studying the case of the end-point estimates for $p=\infty$ concerning  the boundedness of
 some classical operators associated to $\L$
 such as the Littlewood-Paley square functions, fractional integrals   and Riesz transforms
   (see \cite{BHS2008, BHS2009, DY1, DY2, DGMTZ, HLMMY, MSTZ}).

 The following theorem is the main result of this article.

 \begin{thm}\label{th1.1}
 Suppose $V\in B_q$ for some $q\geq n.$
We denote by ${\rm HMO_\L}(\Real_+^{n+1})$  the class of all $C^1$-functions $u(x,t)$
of the solution of ${\mathbb L}u=0$  in $\Real_+^{n+1} $ such that
\begin{eqnarray}\label{e1.8}
\|u\|^2_{{\rm HMO_\L}(\Real_+^{n+1})}&=& \sup_{x_B, r_B}     r_B^{-n} \int_0^{r_B}\int_{B(x_B, r_B)} t | \nabla u(x,t) |^2
{dx dt }  <\infty,
\end{eqnarray}
where $\nabla=(\nabla_x, \partial_t).$ Then we have
 \begin{itemize}
\item[(1)]  If $u\in {\rm HMO_\L}(\Real_+^{n+1})$, then    there exists some $f\in {\rm BMO_\L}(\Real^{n})$ such that $u=e^{-t\sqrt \L}f$,
and
$$
\|f\|_{{\rm BMO_\L}(\RR)}\leq C\|u\|_{{\rm HMO_\L}(\Real_+^{n+1})}
$$
with some constant $C>0$ independent of $u$ and $f$.

\item[(2)]  If $f\in {\rm BMO_\L}(\RR)$, then   the function $u=e^{-t\sqrt \L}f$
satisfies estimates \eqref{e1.8} with
 $
 \|u\|_{{\rm HMO_\L}(\Real_+^{n+1})}\approx \|f\|_{{\rm BMO_\L}(\RR)}.
 $
 \end{itemize}
\end{thm}

We should mention that for the Schr\"odinger operator $\L$ in \eqref{e1.2},
an important property of the $B_q$ class, proved in \cite[Lemma 3]{Ge}, assures that the condition $V\in B_q$ also implies $V\in B_{q+\epsilon}$
for some $\epsilon>0$
 and that the $B_{q+\epsilon}$ constant of $V$ is controlled in terms of the one of $B_q$ membership. This in particular implies
 $V\in L^q_{\rm loc}(\RR)$ for some $q$ strictly greater than $n/2.$ However,  in general the potential $V$ can be unbounded and does not
 belong to $L^p(\RR)$ for any $1 \le p \le \infty .$  As a model example, we could take $V(x)=|x|^2$.
 Moreover, as
  noted in \cite{Shen}, if $V$ is any nonnegative
 polynomial, then $V$ satisfies the stronger condition
 \begin{eqnarray*}
\max_{x\in B} V(x)\leq\frac{C}{\abs{B}}\int_BV(y)~dy,
\end{eqnarray*}
which implies $V\in B_q$ for every $q\in (1, \infty)$ with a uniform constant.

This article is organized as follows. In Section 2, we recall some preliminary results including
the  kernel estimates of the heat   and Poisson semigroups of $\L$,
  the $H^1_{\L}(\RR)$ and ${\rm BMO}_{\L}(\RR)$ spaces associated to the Schr\"odinger operators
and certain properties of ${\mathbb L}$-harmonic functions.
In  Section 3, we will prove our main result,   Theorem~\ref{th1.1}. The proof of part  (1) follows a similar method
to that   of \cite{FJN}, which
depends  heavily on  three non-trivial results:  Alaoglu's Theorem on the weak-$\ast$ compactness
of the unit sphere in the dual of a Banach space,   the duality theorem asserting that the space ${\rm BMO}_{{\mathcal{L}}}(\RR)$ is the dual
space of  $ H^1_{\L}(\RR)$,
 and
some specific properties of the Hardy space $H^1_{\L}(\RR)$ and  the Carleson measure.
For   part (2), we will prove it by  making  use of the full gradient estimates
on the kernel of the Poisson semigroup in the $(x,t)$ variables under the assumption on $V\in B_q$ for some $q\geq n$.
This  improves previously known results (see  \cite{DY2, DGMTZ, HLMMY, MSTZ}) which
characterize   the space ${\rm BMO}_{{\mathcal{L}}}(\RR)$   in terms of Carleson measure
which are only related to   the  gradient in the $t$ variable.
In Section 4,   we will extend the method    for the space ${\rm BMO}_{\L}(\RR)$ in Section 3
 to obtain some generalizations to Lipschitz-type spaces ${\Lambda}_{\L}^{\alpha}(\RR)$
for $\alpha\in (0, 1)$.


Throughout the article, the letters ``$c$ " and ``$C$ " will  denote (possibly different) constants
that are independent of the essential variables.

\vskip 1cm

\section{Basic properties of the heat  and Poisson semigroups of Schr\"odinger operators}
\setcounter{equation}{0}


In this section, we begin by recalling some basic properties of the critical radii function $\rho(x)$ under the assumption \eqref{e1.3} on $V$
(see Section 2, \cite{DGMTZ}).

\begin{lem}\label{le2.1} Suppose $V\in B_q$ for some $q> n/2.$
There exist $C>0$ and $k_0\geq1$ such that for all $x,y\in\Real^n$
\begin{equation}\label{e2.1}
C^{-1}\rho(x)\left(1+\frac{\abs{x-y}}{\rho(x)}\right)^{-k_0}\leq\rho(y)\leq C\rho(x)
\left(1+\frac{\abs{x-y}}{\rho(x)}\right)^{\frac{k_0}{k_0+1}}.
\end{equation}
In particular, $\rho(x)\sim \rho(y)$ when $y\in B(x, r)$ and $r\leq c\rho(x)$.
\end{lem}
It follows from Lemmas 1.2 and 1.8 in \cite{Shen}  that there is a constant $C_0$ such that
for a nonnegative Schwartz class function $\varphi$ there exists a constant $C$ such that
\begin{equation}\label{e2.2}
\int_{\RR} \varphi_t(x-y)V(y)dy\leq
\left\{
\begin{array}{lll}
Ct^{-1}\big({\sqrt{t}\over \rho(x)}\big)^{\delta}\ \ \ &{\rm for}\ t\leq \rho(x)^2,\\[6pt]
C\big({\sqrt{t}\over \rho(x)}\big)^{C_0+2-n}\ \ \ &{\rm for}\ t> \rho(x)^2,
\end{array}
\right.
\end{equation}
where $\varphi_t(x)=t^{-n/2}\varphi(x/\sqrt{t}),$   and
$
\delta =2-\frac{n}{q}>0.
$

Let $\{e^{-t\L}\}_{t>0}$ be the heat  semigroup associated to $\L$:
\begin{eqnarray}\label{e2.3}
 e^{-t\L}f(x)=\int_{\Real^n}{\mathcal K}_t(x,y)f(y)~dy,\qquad f\in L^2(\Real^n),~x\in\Real^n,~t>0.
\end{eqnarray}
From the Feynman-Kac formula, it is well-known that the    kernel ${\mathcal K}_t(x,y)$
of the semigroup
$e^{-t\L}$  satisfies the estimate
\begin{eqnarray}
0\leq {\mathcal K}_t(x,y)\leq h_t(x-y)
\label{e2.4}
\end{eqnarray}
for all $x,y\in\RR$ and $t>0$, where
\begin{equation}\label{e2.5}
h_t(x-y)=\frac{1}{(4\pi t)^{n/2}}~e^{-\frac{\abs{x-y}^2}{4t}}
\end{equation}
is the kernel of the classical heat semigroup
$\set{T_t}_{t>0}=\{e^{t\Delta}\}_{t>0}$ on $\Real^n$.
This estimate can be improved in time when $V\not\equiv 0$ satisfies the reverse H\"older
condition $B_q$ for some $q> n/2$.
 The function $\rho(x)$ arises naturally in this context.

\begin{lem}[see \cite{DGMTZ}] \label{le2.2} Suppose $V\in B_q$ for some $q> n/2.$
For every $N>0$ there exist the constants $C_N$ and $c$ such that for $ x,y\in\Real^n, t >0$,
 \begin{itemize}
\item[(i)]
\begin{equation*}
0\leq {\mathcal K}_t(x,y)\leq C_Nt^{-n/2}e^{-\frac{\abs{x-y}^2}{ct}}\left(1+\frac{\sqrt{t}}{\rho(x)}
+\frac{\sqrt{t}}{\rho(y)}\right)^{-N} \ {\rm and}
\end{equation*}

\item[(ii)] 
\begin{equation*}
 \abs{\partial_t{\mathcal K}_t(x,y)}\leq C_Nt^{-\frac{n+2}{2}}e^{- \frac{\abs{x-y}^2}{ct}}
 \left(1+\frac{\sqrt t}{\rho(x)}+\frac{\sqrt t}{\rho(y)}\right)^{-N} .
\end{equation*}
 \end{itemize}
\end{lem}

Kato-Trotter formula (see for instance   \cite{DZ2002}) asserts that
\begin{align*}
h_t(x-y)- {\mathcal K}_t(x,y)
=  \int_0^t\int_{\Real^n} h_s(x-z)V(z){\mathcal K}_{t-s}(z, y)dzds.
\end{align*}
Then we have the following result.

\begin{lem}[see \cite{DZ2002}]\label{le2.3} Suppose $V\in B_q$ for some $q> n/2.$
There exists a nonnegative Schwartz  function $\varphi$ on
$\Real^n$ such that
\begin{equation*}
\abs{h_t(x-y)-{\mathcal K}_t(x,y)}\leq\left(\frac{\sqrt{t}}{\rho(x)}\right)^\delta\varphi_t(x-y),\quad x,y\in\Real^n,~t>0,
\end{equation*}
where $\varphi_t(x)=t^{-n/2}\varphi\left(x/\sqrt{t}\right)$ and
$
\delta =2-\frac{n}{q}>0.
$
\end{lem}

 The Poisson semigroup associated to $\L$ can be obtained from the heat semigroup \eqref{e2.2} through Bochner's
  subordination formula (see \cite{St1970}):
\begin{align}\label{e2.6}
e^{-t\sqrt{\L}}f(x)&=\frac{1}{\sqrt{\pi}}\int_0^\infty\frac{e^{-u}}{\sqrt{u}}~e^{-{t^2\over 4u}\L}f(x)~du\nonumber\\
&=\frac{t}{2\sqrt\pi}\int_0^\infty
\frac{~e^{-{t^2\over 4s}}}{s^{3/2}}e^{-s\L}f(x) ~ds.
\end{align}

From \eqref{e2.6},  the semigroup kernels ${\mathcal P}_t(x,y)$, associated to $e^{-t\sqrt{\L}}$,
satisfy
the following estimates. For its proof, we refer to \cite[Proposition 3.6]{MSTZ}.

\begin{lem} \label{le2.4} Suppose $V\in B_q$ for some $q> n/2.$ For any $0<\beta<\min\{1, 2-\frac{n}{q}\}$
and every $N>0$, there exists
 a constant $C=C_{N}$ such that
 \begin{itemize}

\item[(i)]
 ${\displaystyle
| {\mathcal P}_t(x,y)|\leq C{t \over (t^2+|x-y|^2)^{n+1\over 2}} \left(1+ {(t^2+|x-y|^2)^{1/2}\over \rho(x)}
+{(t^2+|x-y|^2)^{1/2}\over \rho(y)}\right)^{-N}; }
 $

\item[(ii)] For every $m\in{\mathbb N}$,
\begin{eqnarray*}
|t^m\partial^m_t{\mathcal P}_t(x,y)|\leq C {t^m \over (t^2+|x-y|^2)^{n+m\over 2}} \left(1+ {(t^2+|x-y|^2)^{1/2}\over \rho(x)}
+{(t^2+|x-y|^2)^{1/2}\over \rho(y)}\right)^{-N};
\end{eqnarray*}

\item[(iii)]
${\displaystyle \big|  t\partial_te^{-t\sqrt{\L}}(1)(x)\big|\leq C\left({t\over \rho(x)}\right)^{\beta}\left(1+{t\over \rho(x)}\right)^{-N}.}$
 \end{itemize}
\end{lem}


Recall that a  Hardy-type space associated to $\L$ was introduced by J. Dziuba\'nski et al.  in \cite{ DZ1999, DZ2002, DGMTZ},  defined  by
\begin{equation}\label{e2.7}
 H^1_{\L}(\RR)=\big\{ f\in L^1(\RR): {\mathcal P}^{\ast}f(x)= \sup_{t>0}|e^{-t\sqrt{\L}}f(x)|\in L^1(\RR) \big\}
\end{equation}
with
$$
\|f\|_{H^1_{\L}(\RR)}=\|{\mathcal P}^{\ast}f\|_{L^1(\RR)}.
$$
 For the above class of potentials,   $H^1_{\L}(\RR)$ admits an atomic characterization, where
cancellation conditions are only required for atoms with small supports.
It can be verified that for   every $m\in{\mathbb N}$, for fixed $t>0$ and  $x\in\RR,$
$\partial_t^m {\mathcal P}_t(x, \cdot)\in H^1_{\L}(\RR)$ with
 \begin{equation}\label{e2.8}
 \|\partial_t^m {\mathcal P}_t(x, \cdot)\|_{H^1_{\L}(\RR)}
 \leq Ct^{-m}.
 \end{equation}
 Indeed, from  (ii) of Lemma~\ref{le2.4}  we have that for a fixed $y\in \RR,$
 \begin{eqnarray*}
\sup_{s>0}|e^{-s\sqrt{\L}}\big(t^m\partial_t^m{\mathcal P}_t(\cdot, y)\big)(x)|&\leq& C \sup_{s>0} \left({t^m\over (t+s)^m}
{ (t+s)^m  \over (t+s+|x-y|)^{n+m }}\right) \nonumber\\
&\leq& C
{ t \over (t+ |x-y|)^{n+1 }} \in L^1(\RR, dx),
 \end{eqnarray*}
 which, in combination with the fact that $\partial_t^m {\mathcal P}_t(x, \cdot)=\partial_t^m {\mathcal P}_t(\cdot, x)$,
shows  estimate \eqref{e2.8}.

\begin{lem}\label{le2.5}
Suppose $V\in B_q$ for some $q> n/2.$ Then
 the dual space of $H^1_{\L}(\RR)$ is  ${{\rm BMO}_{\L}(\RR)}$, i.e.,
$$
(H^1_{\L}(\RR))^{\ast}={{\rm BMO}_{\L}(\RR)}.
$$
\end{lem}

 \begin{proof} For the proof, we refer to \cite[Theorem 4]{DGMTZ}. See also \cite{DY2,  HLMMY}.
 \end{proof}

  \smallskip

 In the sequel,
we may sometimes use capital letters to denote points in
${\Bbb R}^{n+1}_+$, e.g.,
$Y=(y, t),$  and set
\begin{eqnarray*}
{ \nabla}_Y u(Y)=(\nabla_y u, \partial_tu) \ \ \ {\rm and}\ \ \
|{ \nabla}_Yu|^2=|\nabla_y u|^2 +|\partial_tu|^2.
\end{eqnarray*}
For simplicity we will denote by
$\nabla$  the full gradient $\nabla_Y$ in $\mathbb{R}^{n+1}$.
We now recall a Moser type local boundedness estimate (see for instance, \cite{HLMMY}) and
include a proof here for the sake of self-containment.

\begin{lemma}\label{le2.6} Suppose $0\leq V\in L^q_{\rm loc}(\RR)$ for some $q> n/2.$
  Let $u$ be a weak solution of ${\mathbb L}u=0$
in
the ball $B(Y_0,2r) \subset \mathbb{R}^{n+1}$.
Then for any $p\geq 1$,   there exists a constant $C=C(n, p)>0$ such that

\begin{eqnarray*}
\sup_{B(Y_0, r)}| u(Y)| \leq C\Big({1\over r^{n+1}}
\int_{B(Y_0, 2r)}| u(Y)|^pdY\Big)^{1/p}.
\end{eqnarray*}
\end{lemma}

\begin{proof} It is enough to show that $u^2$ is a subharmonic function.
Since for any
$\varphi\in C_0^1(B(Y_0,2r))$ with $\varphi\geq 0$,  we have  that
 $
\int  V\varphi u^2 dY\leq C_r\|u\|_{W^{1,2}(B(Y_0,2r))}<\infty
 $
 (see for instance,  \cite[Lemma 3.3]{Gu}).
This gives
\begin{eqnarray*}
\int
&&\hspace{-1cm}{\nabla}u^2\cdot{\nabla}\varphi\,dY\\
&=&2\int
{\nabla}u\cdot{\nabla}(u\varphi)\,dY
-2\int
\varphi|{\nabla}u|^2\,dY\\
&=&-2\int V\varphi u^2\,dY
-2\int\varphi|{\nabla}u|^2\,dY\\
&\leq& 0.
\end{eqnarray*}
\noindent The desired result follows readily.
\end{proof}

\begin{lem}\label{le2.7}
Suppose $ V\in B_q$ for some $q\geq n/2.$
 Assume that  $u\in W_{\rm loc}^{1,2}(\RR)$ is a weak solution of  $(-\Delta+V)u=0$ in $\RR$. Also assume that
there is a $d>0$ such that
\begin{eqnarray}\label{e2.9}
\int_{\RR}{|u(x)|\over 1+|x|^{n+d}} dx\leq C_{d}<\infty.
\end{eqnarray}
Then $u=0$ in $\RR$.
\end{lem}

\begin{proof}
Fix an  $R\geq 10$, we let $\varphi\in C_0^{\infty}(B(0, 3R/4))$ such that $0\leq \varphi\leq 1, \varphi=1$ on $B(0, 5R/8)$,
and $|\nabla \varphi|\leq C/R, |\nabla^2 \varphi|\leq C/R^2.$ Following  \cite[Lemma 6.1]{Shen1999}, one writes
$$
(-\Delta +V)(u\varphi)=-2\nabla u\cdot \nabla\varphi -u\Delta\varphi \ \ \ {\rm in}\ \RR.
$$
Then we have
$$
u(x)\varphi(x)=\int_{\RR}\Gamma_V(x,y)\{-2\nabla u\cdot \nabla\varphi -u\Delta\varphi\} dy,
$$
where $\Gamma_V(x,y)$ denotes  the  fundamental solution of $-\Delta+V$ in $\RR$.
Hence,  for any $x\in B(0, R/2)$,
\begin{eqnarray}\label{e2.10}
|u(x)|&\leq& {C\over R}\int_{5R/8\leq |y|\leq 3R/4} | \Gamma_V(x,y)| \left(|\nabla u(y)| +{|u(y)|\over R}\right) dy  \nonumber\\
&\leq& {C\over R^2}\left\{ \int_{5R/8\leq |y|\leq 3R/4}  |\Gamma_V(x,y) |^2 dy\right\}^{1/2}
 \left\{ \int_{B(0, R)}  |u(y) |^2 dy\right\}^{1/2},
 \end{eqnarray}
 where we have used the H\"older inequality and Caccioppoli's inequality.

From the upper bound of $\Gamma_V(x,y)$ in \cite[Theorem 2.7]{Shen}, we have that for every $k>d$ and every  $x\in B(0, R/2)$,
 \begin{eqnarray}\label{e2.11}
 \int_{5R/8\leq |y|\leq 3R/4}  |\Gamma_V(x,y) |^2 dy  &\leq&   C_k\int_{5R/8\leq |y|\leq 3R/4}
 \bigg|{1\over \left(1+ {|x-y|\over \rho(x)}\right)^k } {1\over |x-y|^{n-2}} \bigg|^2 dy\nonumber\\
 &\leq& C_k \rho(x)^{2k} R^{-2k-n+4}.
 \end{eqnarray}
  Recall that the condition $B_{n/2}$ implies $V\in B_{q_0}$ for some $q_0>n/2$.
  It follows from Lemma 2.9 of \cite{Shen1999}   that
 $
 |u(y)|\leq |B|^{-1}\int_B|u(z)|dz$ if $ B=B(y, R)\subset \RR
$
(see also Lemma~\ref{le2.6}). Then we have
 \begin{eqnarray*}
\int_{B(0, R)}  |u(y) |^2 dy   &\leq&  C R^n \left( {C\over R^n}\int_{B(0, 2R)} |u(y)|dy \right)^2
\leq  C{ R^{n+2d}}\left(\int_{\RR} {|u(y)|\over 1+|y|^{n+d}}dy \right)^2
\leq  C C_d^2{ R^{n+2d}}.
 \end{eqnarray*}
 This, in combination with \eqref{e2.10} and \eqref{e2.11}, yields that for every $x\in B(0, R/2)$,
\begin{eqnarray*}
|u(x)|&\leq& C'_k C_{d}\,  \rho(x)^k R^{d-k}.
 \end{eqnarray*}
Letting $R\to +\infty$, we obtain that $u(x)=0$ and therefore, $u=0$  in the whole $\RR.$ The proof is complete.
\end{proof}

\begin{rem}For any $d\geq 0$,  one writes
\begin{eqnarray*}
{ {\mathcal H}_d(\L)}=\Big\{f\in W^{1,2}_{\rm loc}({\mathbb R}^n):
\L f =0 \ {\rm and}\   \ |f(x)|=O(|x|^d) \ \ {\rm as}\ |x|\rightarrow \infty\Big\}
\end{eqnarray*}
and
\begin{eqnarray*}
{ {\mathcal H}_{\L}}=\bigcup_{d:\ 0\leq d<\infty}
{ {\mathcal H}_d(\L)}.
\end{eqnarray*}
By  Lemma~\ref{le2.7}, it follows that
for any $d\geq 0$,
$$
{ {\mathcal H}_{\L}}={ {\mathcal H}_d(\L)}=\big\{0\big\}.
$$
See also   Proposition 6.5 of \cite{DY2}.
\end{rem}

\medskip

At the end of this  section, we establish the following  characterization of Poisson integrals of Schr\"odinger operators
with functions in $L^p(\RR), 1\leq p<\infty.$

\begin{prop}\label{prop2.8}
  Suppose $V\in B_q$ for some $q\geq (n+1)/2.$ If $u$ is a continuous weak solution of
  ${\mathbb L}u =0$ in ${\mathbb R}^{n+1}_+$ and there exist
 a constant $C>0$ and a $1\leq p< \infty$, such that
 $$
 \|u(\cdot,t)\|_{L^p(\RR)}=\left(\int_{\RR} |u(x,t)|^p dx\right)^{1/p}\leq C<\infty
 $$
 for all $t>0$, then

 \begin{itemize}
\item[(1)]  when $1<p< \infty$, $u(x,t)$ is the Poisson  integral of
a function $f$ in $L^p(\RR);$

\item[(2)]  if $p=1$, $u(x,t)$ is the Poisson-Stieltjes integral of a finite Borel measure; if,
in addition, $u(\cdot, t)$ is Cauchy in the $L^1$ norm as $t>0$ then $u(x,t)$ is the Poisson  integral of
a function $f$ in $L^1(\RR).$
 \end{itemize}
 \end{prop}

 \begin{proof} The proof of Proposition~\ref{prop2.8} is standard (see for instance, Theorem 2.5, Chapter 2 in \cite{SW}).
 We give a brief argument of this proof for completeness
 and the convenience of the reader.

  If $1<p< \infty$ and $ \|u(\cdot,t)\|_{L^p(\RR)}\leq C<\infty$
for all $t>0$, then there exists a sequence $\{t_k\}$ such that $\lim\limits_{k \rightarrow \infty} t_k=0$, and a function $f\in L^p(\RR)$
such that $u(\cdot, t_k)$ converges weakly to $f$ as $k\rightarrow \infty.$ That is, for each $g\in L^{p'}(\RR), 1/p+1/p'=1,$
$$
\lim\limits_{k\rightarrow \infty} \int_{\RR} u(y, t_k) g(y)dy =\int_{\RR} f(y) g(y)dy.
$$
If $p=1$ there exists a finite Borel measure $\mu$ that is the weak-$\ast$ limit of a sequence $\{ u(\cdot, t_k)\}$. That is,
for each $g$ in $C_0(\RR)$,
$$
\lim\limits_{k\rightarrow \infty} \int_{\RR} u(y, t_k) g(y)dy =\int_{\RR}  g(y)d\mu(y).
$$
For any $x\in\RR, t>0$, we take $g(y)=p_t(x,y)\in L^{p'}(\RR)$ for $1\leq p'\leq \infty$ and also belongs to $C_0(\RR)$
we have, in particular,
$$
\lim\limits_{k\rightarrow \infty} \int_{\RR} p_t(x,y)u(y, t_k) dy =\int_{\RR} p_t(x,y) f(y)  dy
$$
when  $1<p< \infty$, and
$$
\lim\limits_{k\rightarrow \infty} \int_{\RR} p_t(x,y) u(y, t_k)  dy =\int_{\RR}  p_t(x,y) d\mu(y)
$$
when $p=1.$

 Since $u$ is continuous, $\lim\limits_{t\rightarrow 0^+} u(x, t+t_k)=u(x, t_k)$.
 It is well known that if  $u(\cdot, t_k)\in L^p(\RR), 1\leq p<\infty$,
then  $\lim\limits_{t\rightarrow 0^+} e^{-t\sqrt{\L}}(u(\cdot, t_k))=u(x, t_k)$ for almost every $x\in \RR.$
Set $w(x,t)=e^{-t\sqrt \L}(u(\cdot, t_k))(x)-u(x,t+t_k)$.  The function  $w$ satisfies ${\mathbb L}w =0, w(x,0)=0$.
Define,
 \begin{align*}
{\overline w}(x,t)=\left\{
 \begin{array}{rrl}
  w(x,t),&\ \ \ &t\ge 0,\\[6pt]
  -w(x,-t),& \ \ &t<0.
 \end{array}
 \right.
 \end{align*}
Then ${\overline w}$ satisfies
 $$
{\overline {\mathbb L}} {\overline w}(x,t)=0, \ \ \ \ \ \ (x,t)\in  {\mathbb R}^{n+1},
 $$
where  ${\overline {\mathbb L}}$ is an extension operator of ${ {\mathbb L}}$ on $\Real^{n+1}$.
Observe that if $V(x)\in B_q(\Real^{n})$ with $q\ge (n+1)/2,$ then it can be verified
that  $V(x,t)=V(x)\in B_q$ on $\Real^{n+1}$. To apply Lemma \ref{le2.7}, we need to verify \eqref{e2.9}. Indeed,
\begin{eqnarray*}
\int_{{\mathbb R}^{n+1}} {|{\overline w}(x,t)|\over 1+  |(x,t)|^{n+3}}dxdt&\leq& C\int_0^{\infty}  {1\over 1+t^{2}} \left(
\int_{\RR} {|w(x,t)|\over 1+  |x|^{n+{1}}}dx\right) dt \\
&\leq& C\int_0^{\infty}  {1\over 1+t^{2}} \left(\sup_{t>0} \|w(\cdot, t)\|_p  \right) dt \\
&\leq& C\sup_{t>0} \left(\int_{\RR} |u(x,t)|^p dx\right)^{1/p} <\infty,
\end{eqnarray*}
and so \eqref{e2.9} holds. By Lemma \ref{le2.7}, we have that $\overline w\equiv0$, and then $w=0$, that is,
 $$u(x,t+t_k)=e^{-t\sqrt \L}(u(\cdot, t_k))(x), \ \ \  x\in\RR, \ t>0.$$
Therefore,
$$
u(x,t)=\lim\limits_{k\rightarrow \infty}
u(x,t+t_k) =\lim\limits_{k\rightarrow \infty}
e^{-t\sqrt \L}(u(\cdot, t_k))(x) =e^{-t\sqrt{\L}}f(x)
$$
when $1< p< \infty$, and $u(x,t)=\int_{\RR}  p_t(x,y) d\mu(y) $ when $p=1.$
The proof is complete.
 \end{proof}

 \medskip

\section{Proof of  the Main Theorem }
\setcounter{equation}{0}

\subsection{Existence  of boundary values of ${\mathbb L}$-harmonic functions}
\begin{lem}\label{le3.1} For every  $u\in {\rm HMO_\L}(\Real_+^{n+1})$ and
 for every $k\in{\mathbb N}$, there exists a constant $C_k>0$ such that
\begin{equation*} \label{dd}
\int_{\RR}{|u(x,{1/k})|^2\over (1+|x|)^{2n}}  dx\leq C_k <\infty,
\end{equation*}
hence $u(x, 1/k)\in L^2((1+|x|)^{-2n}dx)$. Therefore for  all $k\in{\mathbb N}$, $e^{-t\sqrt{\L}}(u(\cdot, {1/k}))(x)$ exists
everywhere in ${\mathbb R}^{n+1}_+$.
\end{lem}

\begin{proof}  Since $u\in C^{1}({\mathbb R}^{n+1}_+)$, it  reduces to show that for every $k\in{\mathbb N},$
\begin{eqnarray}\label{e3.1}
\int_{|x|\geq 1} {|u(x,{1/k})- u(x/|x|, 1/k) |^2 \over (1+|x|)^{2n}} dx\leq C_k\|u\|^2_{{\rm HMO_\L}(\Real_+^{n+1})}<\infty.
 \end{eqnarray}
To do this, we write
\begin{align*}
&\hspace{-0.3cm}u(x, 1/k)- u(x/|x|, 1/k)\\&=\big[u(x, 1/k)- u(x, |x|)\big]
 +\big[u(x, |x|)-u(x/|x|, |x|)\big] +\big[u(x/|x|, |x|)-u(x/|x|, 1/k)\big].
 \end{align*}
Let
  \begin{eqnarray*}
 I=  \int_{|x|\geq 1 } {|u(x, 1/k)- u(x, |x|) |^2 \over (1+|x|)^{2n}} dx,
 \end{eqnarray*}
  \begin{eqnarray*}
 II=  \int_{|x|\geq 1 } {|u(x, |x|)-u(x/|x|, |x|) |^2 \over (1+|x|)^{2n}} dx,
 \end{eqnarray*}and
  \begin{eqnarray*}
 III=  \int_{|x|\geq 1 } {|u(x/|x|, |x|)-u(x/|x|, 1/k) |^2 \over (1+|x|)^{2n}} dx.
 \end{eqnarray*}

 Set $Y_0=(x,t)$ and $r=t/4$. We use  Lemma~\ref{le2.6} for $\partial_t u$ and Schwarz's inequality to obtain
 \begin{eqnarray}\label{e3.2}
 \big| \partial_t u(x, t)\big| &\leq& C\Big({1\over r^{n+1}} \int_{B(Y_0, 2r)}
 | \partial_s u(y, s) |^2 {dY } \Big)^{1/2}\nonumber\\
  &\leq& C\Big({1\over t^{n+1}} \int_{B(x, t/2)}\int_{t/2}^{3t/2}
 | \partial_s u(y, s) |^2 {dsdy } \Big)^{1/2}\nonumber\\
  &\leq& C t^{-1}\Big({1\over |B(x, 2t)|} \int_{0}^{2t}\int_{B(x, 2t)}
  s|  \partial_s u(y, t) |^2 {dyds } \Big)^{1/2}
  \nonumber\\
  &\leq&  C  t^{-1}\|u\|_{{\rm HMO_\L}(\Real_+^{n+1})},
  \end{eqnarray}
which gives
 \begin{eqnarray}\label{e3.3}
 | u(x, |x|)-u(x, 1/k) |  =  \Big| \int_{1/k}^{|x|}   \partial_t u(x, t) dt \ \Big| \leq C\log(k|x|) \|u\|_{{\rm HMO_\L}(\Real_+^{n+1})}.
 \end{eqnarray}
It follows that
   \begin{eqnarray*}
 I+III
 &\leq& C\|u\|^2_{{\rm HMO_\L}(\Real_+^{n+1})} \int_{|x|\geq 1 } {1\over (1+|x|)^{2n}}  \log^2 (k |x|)     dx  \\
 &\leq& C_k  \|u\|^2_{{\rm HMO_\L}(\Real_+^{n+1})}.
  \end{eqnarray*}

For the term $II,$  we  have that for any $x\in \RR,$
 $$
  u(x, |x|)- u(x/|x|, |x|)  =\int_1^{|x|}  D_r u(r\omega, |x|)  dr, \ \ \ \ x=|x| \omega.
 $$
 Let $B=B(0, 1)$  and  $2^mB=B(0, 2^m)$.
Note that for every $m\in{\mathbb N}$,  we have
  \begin{align*}
   \int_{2^mB\backslash 2^{m-1}B} \left| \int_1^{|x|}\left|    D_r  u(r\omega, |x|) \right|   dr \right|^2 dx
  &=     \int_{2^{m-1}}^{2^{m}} \int_{|\omega|=1} \left| \int_1^{\rho}     D_r u(r\omega, \rho)    dr \right|^2 \rho^{n-1}  d\omega d\rho   \nonumber\\
      &\leq   2^{mn} \int_{2^{m-1}}^{2^{m}} \int_{|\omega|=1}\int_1^{ 2^{m}}    |  D_r u(r\omega, \rho)  |^2 dr d\omega d\rho   \nonumber\\
	  &\leq    2^{mn} \int_{2^{m-1}}^{2^{m}} \int_{2^mB\backslash B}   |  \nabla_y u(y, t)  |^2 |y|^{1-n} dy dt\nonumber\\
	  &\leq    2^{mn} \int_{2^{m-1}}^{2^{m}} \int_{2^mB}   |  \nabla_y u(y, t)  |^2  dy dt,
  \end{align*}
  which gives
  \begin{align*}
 \int_{2^mB\backslash 2^{m-1}B}   | u(x, |x|) - u(x/|x|, |x|) |^2    dx
    &\leq     C2^{m(2n-1)} \left( {1\over |2^mB|}  \int_{0}^{2^{m}} \int_{2^mB }    | t \nabla_y u(y, t)  |^2
	{dydt\over t}   \right) \nonumber
	 \\
    &\leq  C  2^{m(2n-1)}\|u\|^2_{{\rm HMO_\L}(\Real_+^{n+1})}.
  \end{align*}
Therefore,
 \begin{align*}
II
  &\leq C \sum_{m=1}^{\infty}  {1\over 2^{2mn}}
   \int_{2^mB\backslash 2^{m-1}B}    | u(x, |x|) - u(x/|x|, |x|) |^2    dx\leq C\|u\|^2_{{\rm HMO_\L}(\Real_+^{n+1})}.
  \end{align*}
  Combining estimates of $I, II$ and $III$, we have obtained \eqref{e3.1}.

  Note that by Lemma~\ref{le2.4},  if
  $V\in B_q$ for some $q\geq n/2$, then  the semigroup kernels ${\mathcal P}_t(x,y)$, associated to $e^{-t\sqrt{\L}}$,
  decay
faster than any power of $1/|x-y|$.
Hence,    for  all $k\in{\mathbb N}$, $e^{-t\sqrt{\L}}(u(\cdot, {1/k}))(x)$ exists
everywhere in ${\mathbb R}^{n+1}_+$.
This completes the proof.
  \end{proof}

\begin{lem}\label{le3.2} For every  $u\in {\rm HMO_\L}(\Real_+^{n+1})$,
we have that for every $k\in{\mathbb N}$,
\begin{equation*}
u(x, t+{1/k})=e^{-t\sqrt{\L}}\big(u(\cdot, {1/k})\big)(x), \ \ \ \ x\in\RR, \  t>0.
\end{equation*}
\end{lem}

 \begin{proof}
 Since $u(x, \cdot)$ is continuous on $\Real_+$, we have that $\lim_{t\to 0^+}  u(x,t+{1}/{k})= u(x,{1}/{k}).$
Let us first show that for every $k\in{\mathbb N}$,
 \begin{equation}\label{e3.4}
 \lim_{t\to 0^+}e^{-t\sqrt \L}(u(\cdot, 1/k))(x)=u(x,1/k),  \ \   x\in \Real^n.
\end{equation}
One writes
\begin{eqnarray*}
e^{-t\sqrt \L}(u(\cdot, 1/k) )(x)&=&
  e^{-t\sqrt \L}  (u(\cdot, 1/k)1\!\!1_{|x-\cdot|> 1})(x)\\[1pt]
 & &+
\big(e^{-t\sqrt \L}- e^{-t\sqrt {-\Delta}}\big) (u(\cdot, 1/k)1\!\!1_{|x-\cdot|\leq 1})(x)+
 e^{-t\sqrt {-\Delta}}  (u(\cdot, 1/k)1\!\!1_{|x-\cdot|\leq 1})(x) \\[1pt]
&= &
  I(x,t)+II(x,t)+III(x,t).
\end{eqnarray*}
 By Lemma~\ref{le3.1}, we have that $u(x, 1/k)\in L^1((1+|x|)^{-2n}dx)$.
From Lemma~\ref{le2.4}, estimates of  the   semigroup kernels ${\mathcal P}_t(x,y)$, associated to $e^{-t\sqrt{\L}}$,   show  that
\begin{eqnarray*}
I(x,t)&\leq&    Ct \rho(x)^n\int_{|x-y|> 1} {1\over 1+|x-y|^{2n+1}}|u(y, 1/k)|   dy
 \leq   Ct\rho(x)^n (1+|x|^{2n})\int_{\RR} {1\over 1+| y|^{2n}}|u(y, 1/k)|   dy ,
\end{eqnarray*}
hence $
 \lim_{t\to 0^+}I(x,t)=0.$

For the term $II(x,t)$, it follows from Lemma~\ref{le2.3}  that  there exists a nonnegative Schwartz  function $\varphi$
 such that
 \begin{eqnarray*}
II(x,t) &\leq&  C  \int_0^ {\infty}\int_{|x-y|\leq 1} {t\over s^{3/2}}
 \exp \big(-{t^2\over 4s}\big)    | {\mathcal K}_s(x,y)-h_s(x-y)|  |u(y, 1/k)| dy  ds\\
&  \leq& C \| u(\cdot, 1/k)  \|_{L^{\infty}( B(x, 1))}\bigg[\int_0^{\rho(x)^2} \int_{\RR} {t\over s^{3/2}}
  \exp \big(-{t^2\over 4s}\big)  \Big({\sqrt{s}\over \rho(x)}\Big)^{\delta}
 \varphi_s(x-y)    dy  ds \\
&& \hspace{4.3cm}  +\int_{\rho(x)^2}^{\infty} \int_{\RR} {t\over s^{3/2}}
  \exp \big(-{t^2\over 4s}\big)   h_s(x-y)
  dy  ds\bigg]\\
&\leq&C t \rho(x)^{-1}\| u(\cdot, 1/k)  \|_{L^{\infty}( B(x, 1))},
\end{eqnarray*}
which implies that
$
 \lim_{t\to 0^+}II(x,t)=0.$
  Finally, we can follow a standard  argument as in the proof of Theorem 1.25, Chapter 1 of \cite{SW} to
 show  that for every  $x\in \Real^n,$
$
\lim_{t\to 0^+}e^{-t\sqrt {-\Delta}}  (u(\cdot, 1/k)1\!\!1_{|x-\cdot|\leq 1})(x)= u(x,{1}/{k}),
$
and so \eqref{e3.4} holds.

Next,  we follow an argument as in Proposition~\ref{prop2.8} to
set $w(x,t)=e^{-t\sqrt \L}(u(\cdot, 1/k))(x)-u(x,t+1/k)$.  The function  $w$ satisfies ${\mathbb L}w =0, w(x,0)=0$.
Define,
 \begin{align*}
{\overline w}(x,t)=\left\{
 \begin{array}{rrl}
  w(x,t),&\ \ \ &t\ge 0,\\[6pt]
  -w(x,-t),& \ \ &t<0.
 \end{array}
 \right.
 \end{align*}
Then ${\overline w}$ satisfies
 $$
{\overline {\mathbb L}} {\overline w}(x,t)=0, \ \ \ \ \ \ (x,t)\in  {\mathbb R}^{n+1},
 $$
where  ${\overline {\mathbb L}}$ is an extension operator of ${ {\mathbb L}}$ on $\Real^{n+1}$.
Observe that if $V(x)\in B_q(\Real^{n})$ with $q\ge n,$ then it can be verified
that  $V(x,t)=V(x)\in B_q$ on $\Real^{n+1}$.
Next, let us verify \eqref{e2.9}. One writes
\begin{eqnarray*}
\int_{{\mathbb R}^{n+1}} {|{\overline w}(x,t)|\over 1+ |(x,t)|^{4(n+1)}}dxdt
& \leq& 2\int_{{\mathbb R}^{n+1}_+} {|e^{-t\sqrt \L}(u(\cdot, 1/k))(x)|\over 1+ (|x|^2+t^2)^{2(n+1)}}dxdt \\
&& +2\int_{{\mathbb R}^{n+1}_+} {|u(x, 1/k)|\over 1+ (|x|^2+t^2)^{2(n+1)}}dxdt \\
&& +2\int_{{\mathbb R}^{n+1}_+} {|u(x, t+1/k)-u(x, 1/k)|\over 1+ (|x|^2+t^2)^{2(n+1)}}dxdt=IV+V+VI. \\
\end{eqnarray*}
Observe that   if $t\geq 1,$ then  by
  Lemmas~\ref{le2.1} and ~\ref{le2.4},
\begin{eqnarray*}
 \int_{\RR} {|e^{-t\sqrt \L}(u(\cdot, 1/k))(x)|\over (1 + |x|^2)^{2n}}   dx &\leq& C\int_{\RR} \left(\int_{\RR}{1\over (1+ |x|)^{4n}}
 {t\, \rho(x)^{2n}\over   (t+|x-y|)^{3n+1}} dx \right) |u(y,1/k)|dy    \\
 &\leq& C\int_{\RR} \left(\int_{\RR}{1\over (1+ |x|)^{2n}}
 {t \over   (t+|x-y|)^{3n+1}} dx \right) |u(y,1/k)|dy    \\
&\leq& C (t+1) \int_{\RR} {|u(y,1/k)|\over  1+ |y|^{2n} } dy\leq   C_k(t+1),
\end{eqnarray*}
which gives
\begin{eqnarray*}
IV \leq  C\int_0^{\infty} {1\over (1+t)^4} \left(  \int_{\RR} {|e^{-t\sqrt \L}(u(\cdot, 1/k))(x)|\over (1 + |x|^2)^{4n}}   dx\right) dt
\leq C_k\int_0^{\infty} {1\over (1+t)^2} dt \leq   C'_k.
\end{eqnarray*}
If $t\leq  1,$ then $IV \leq C_k$ can be verified  by using condition $u\in C^1({\mathbb R}^{n+1}_+)$
and Lemmas~\ref{le2.1} and ~\ref{le2.4}.
By Lemma~\ref{le3.1}, we have that $V\leq C_k$.  For term $VI$, we use  \eqref{e3.2} to obtain that $  | \partial_t u(x, t) |
  \leq   C /t,
$ and then
 \begin{eqnarray*}
 | u(x, t+1/k)-u(x, 1/k) |  =  \Big| \int_{1/k}^{t+1/k}   \partial_s u(x, s) ds \ \Big| \leq C\log(1+kt),
 \end{eqnarray*}
which gives
 \begin{eqnarray*}
VI
& \leq& C \int_{{\mathbb R}^{n+1}_+} {\log(1+kt) \over 1+ (|x|^2+t^2)^{2(n+1)}}dxdt
\leq C_k.
\end{eqnarray*}
  Estimate \eqref{e2.9} then follows readily.

By Lemma \ref{le2.7}, we have that $\overline w\equiv0$, and then $w=0$, that is,
 $$u(x,t+1/k)=e^{-t\sqrt \L}(u(\cdot, 1/k))(x), \ \ \  x\in\RR, \ t>0.$$
The proof is complete.
\end{proof}

\begin{rem}\label{re3.3} Suppose $V\in B_q$ for some $q\geq n/2$ and let $\L=-\Delta +V.$
Using a similar argument as in \eqref{e3.4}, we have  that for every $f\in L^{\infty}(\RR)$,
 \begin{itemize}
\item[(i)] $\lim\limits_{t\rightarrow 0^+}e^{-t\L}  f(x)= f(x),$ \ a.e.  $x\in{\RR};$

\smallskip

\item[(ii)] $\lim\limits_{t\rightarrow 0^+}e^{-t\sqrt{\L}}  f(x)= f(x), \ $ a.e. ${x\in\RR.}$
\end{itemize}
For the heat and Poisson integrals of the harmonic oscillator $  \L= -\Delta+|x|^2$ on $\RR$, we refer to  Remarks 2.9-2.11,  \cite{ST}.
See also \cite{M}.

From (ii), it follows from an argument as in Proposition~\ref{prop2.8}   that  for  $V\in B_q$ for some $q\geq (n+1)/2$,
if $u$ is a continuous weak solution of
  ${\mathbb L}u =0$ in ${\mathbb R}^{n+1}_+$ with
 $
 \sup_{t>0}\|u(\cdot,t)\|_{L^{\infty}(\RR)} \leq C<\infty
  $, then   $u(x,t)$ is the Poisson  integral of
a function $f$ in $L^{\infty}(\RR).$

\end{rem}

From now on, for any $k\in{\mathbb N}$, we set
$$
u_k(x,t)= u(x, t+{1/k} ).
$$
Following an argument as in \cite[Lemma 1.4]{FJN}, we have

\begin{lem}\label{le3.4}For every  $u\in {\rm HMO_\L}(\Real_+^{n+1})$,
 there exists a constant $C>0$ (depending only on $n$) such that for all  $k\in{\mathbb N},$
\begin{equation}\label{e3.5}
\sup_{x_B, r_B}   r_B^{-n}\int_0^{r_B}\int_{B(x_B, r_B)} t | \partial_t u_k(x,t)|^2 {dx dt}
  \leq C\|u\|^2_{{\rm HMO_\L}(\Real_+^{n+1})}<\infty.
\end{equation}
\end{lem}

 \begin{proof}
Let $B=B(x_B, r_B)$. If $r_B\geq 1/k$, then letting $s=t+1/k$, it follows that
 \begin{align*}
 \int_0^{r_B}\int_B t |   \partial_t u (x,t+1/k)|^2 {dx dt }
 &\leq  C  \int_0^{2r_B}\int_{2B} s| \partial_s u (x,s)|^2 {dx ds }\\
 &\leq C |B|\|u\|^2_{{\rm HMO_\L}(\Real_+^{n+1})}<\infty .
\end{align*}
If $r_B< 1/k$, then it follows from Lemma~\ref{le2.6} for   $\partial_{t}u(x, t+{1/k})$ and
 a similar argument as in \eqref{e3.2}   that
 \begin{equation*}
 \big| \partial_{t}u(x, t+{1/k})\big| \leq C \big(t+k^{-1}\big)^{-1} \|u\|_{{\rm HMO_\L}(\Real_+^{n+1})}.
  \end{equation*}
 Therefore,
  \begin{align*}
 |B|^{-1}\int_0^{r_B}\int_B t| \partial_t u (x,t+1/k)|^2 {dx dt }
 &\leq C  |B|^{-1}\|u\|^2_{{\rm HMO_\L}(\Real_+^{n+1})}   \int_0^{r_B}\int_{B} t \big(t+k^{-1}\big)^{-2} {dx dt}\\
  &\leq C  \|u\|^2_{{\rm HMO_\L}(\Real_+^{n+1})} \,  \big(k^2 \int_0^{r_B}  t  { dt}\big)\\
 &\leq C \|u\|^2_{{\rm HMO_\L}(\Real_+^{n+1})}<\infty
\end{align*}
since
$r_B< 1/k.$

By taking the supremum over all balls $B\subset\Real^n,$ we complete the proof of \eqref{e3.5}.
\end{proof}

Letting $f_k(x)=u(x, 1/k), k\in{\mathbb N}$, it follows from Lemma \ref{le3.2}     that
$$
u_k(x,t)=e^{-t\sqrt{\L}}(f_k)(x), \ \ \ x\in\RR, \ t>0.
$$
Recall that a
measure $\mu$ defined on ${\mathbb R}^{n+1}_+$ is said to  be a Carleson measure  if there is a positive constant
$c$ such  that for each ball $B$ on ${\mathbb R}^{n}$,
\begin{equation}\label{e3.6}
\mu({\widehat B})\leq c|B|,
\end{equation}
where ${\widehat B}$ is the tent over $B$.
The smallest bound $c$  in (\ref{e3.6}) is defined to  be the norm of
$\mu$, and is denoted by
$
|||\mu|||_{car}$.

It follows from
Lemma~\ref{le3.4}   that
$$
 \mu_{\nabla_t, f_k}(x,t)=| t\partial_te^{-t\sqrt{\L}}(f_k)(x)|^2 {dx dt\over t}
$$
is a Carleson measure with $||| \mu_{\nabla_t, f_k}|||_{car} \leq C\|u\|^2_{{\rm HMO_\L}(\Real_+^{n+1})}$.

\begin{lem} \label{le3.5} For every  $u\in {\rm HMO_\L}(\Real_+^{n+1})$,
there exists a constant $C>0$ independent of $k$ such that
\begin{equation*}
\|f_k\|_{\rm BMO_\L(\RR)}\leq C<\infty,\ \  \hbox{ for any } k\in \mathbb N.
\end{equation*}
Hence for  all $k\in{\mathbb N}$, $f_k$ is uniformly bounded in  ${\rm BMO_\L(\Real^{n})}$.
\end{lem}

 The proof of Lemma~\ref{le3.5} was given in \cite[Theorem 2]{DGMTZ}; see also  \cite{DY2,   HLMMY, MSTZ}.
These arguments depend  on three  non-trivial results:    the duality theorem that
${\rm BMO}_{{\mathcal{L}}}(\RR)  = \big( H^1_{\L}(\RR)\big)^{\ast}$,
Carleson inequality on tent spaces (see \cite[Theorem 1]{CMS})
 and
some special properties of the space $H^1_{\L}(\RR)$.
In the sequel, we are going to give a direct proof of Lemma~\ref{le3.5} which is independent of these  results such as the duality of
 $ H^1_{\L}(\RR)$ and ${\rm BMO}_{{\mathcal{L}}}(\RR)$ mentioned above (we thank Jie Xiao for this observation).

To prove  Lemma~\ref{le3.5}, we need to  establish the following Lemmas~\ref{le3.6} and \ref{le3.7}.

Given a function
$f\in L^2((1+|x|)^{-2n}dx)$ and an $L^2$ function  $g$  supported on a ball $B=B(x_B, r_B)$, for
 any $ (x,t)\in {\mathbb R}^{n+1}_+$, set
\begin{eqnarray}\label{e3.7}
 F(x,t)=  t\partial_t e^{-t\sqrt \L}   f(x)\ \ \ {\rm and}\ \ \
G(x,t)= t\partial_t e^{-t\sqrt \L}  (I-e^{-r_B\sqrt \L} )g(x).
\end{eqnarray}


\begin{lem} \label{le3.6}
   Suppose $f, g, F, G$ are as in (\ref{e3.7}).
 If $f$   satisfies
$$
||| \mu_{\nabla_t, f}|||^2_{car}=\sup_{x_B, r_B}
 r_B^{-n}\int_0^{r_B}\int_{B(x_B, r_B)} t|  \partial_t  e^{-t\sqrt \L}   f(x)|^2 {dxdt }
 <\infty,
$$
  then there exists a constant $C>0$  such that
\begin{eqnarray}
\int_{{\mathbb R}^{n+1}_+}
|F(x,t) G(x,t)|{dxdt\over t}\leq C |B|^{1/2}||| \mu_{\nabla_t, f}|||_{car}
\|g\|_{{ L}^2(B)}.
\label{e3.8}
\end{eqnarray}
\end{lem}

 \begin{proof}
 To prove
\eqref{e3.8},  let us consider  the square function
${\mathcal G}f$  given  by
$$
{\mathcal G}(f)(x)=\Big(\int_0^{\infty}
|t\partial_t e^{-t\sqrt \L}f(x)|^2{dt\over t}\Big)^{1/2}.
$$
By the spectral theory, the function ${\mathcal G}(f)$ is bounded on
$L^2(\RR)$. Now, given a ball $B=B(x_B, r_B)\subset{\mathbb R}^n$ with radius $r_B$, we  put
$$
T(B)=\{(x,t)\in{\mathbb R}^{n+1}_+: x\in B, \ 0<t<r_B\}.
$$
We then write
\begin{align*}
\int_{{\mathbb R}^{n+1}_+}
&|F(x,t) G(x,t)|{dxdt\over t} \\&=\int_{T(2B)} \big|F(x,t)
G(x,t)\big|{dxdt\over t}+\sum_{k=2}^{\infty}\int_{T(2^{k}B)\backslash T(2^{k-1}B) }
\big|F(x,t) G(x,t)\big|{dxdt\over t}\\
&={\rm A_1} + \sum_{k=2}^{\infty} {\rm A_k}.
\end{align*}
Using the H\"older inequality and $L^2$-boundedness of ${\mathcal G}$,   we obtain
\begin{align*}
 {\rm A_1}
&\leq \Big\|\Big\{\int_0^{2r_B} |  t\partial_t e^{-t\sqrt \L}   f(x)  |^2{dt\over t}\Big{\}}^{1/2}\Big\|_{L^2(2B)}
\|{\mathcal G} ({
{I}}- e^{-r_B\sqrt \L} )g\|_{{ L}^2(\RR)}\\
&\leq C r_B^{n \over 2}||| \mu_{\nabla_t, f}|||_{car}
\| ({
{I}}- e^{-r_B\sqrt \L} )g\|_{{ L}^2(\RR)}\\
&\leq C r_B^{n \over 2}||| \mu_{\nabla_t, f}|||_{car}\| g\|_{{ L}^2(B)}.
\end{align*}

Let us estimate ${\rm A_k}$ for $k=2,3, \cdots.$
Observe that
\begin{align*}
{\rm A}_k
&\leq \Big\|\Big\{\int_0^{2^kr_B}\big|t\partial_t e^{-t\sqrt \L} f(x)\big|^2{dt\over t}\Big\}^{1/2}\Big\|_{ L^2(2^kB)}\\
&\quad \quad \times \Big\|
\Big\{\int_0^{2^kr_B}\big|t\partial_t e^{-t\sqrt \L}  (I-e^{-r_B\sqrt \L} )g(x)\chi_{T(2^{k+1}B)\backslash T(2^{k}B) }\big|^2
{dt\over t}\Big\}^{1/2}\Big\|_{{ L}^{2}(2^kB)}\\
&\leq C (2^kr_B)^{n \over2} ||| \mu_{\nabla_t, f}|||_{car}\times  {\rm B}_k,
\end{align*}
where
\begin{eqnarray*}
{\rm B}_k =\Big{\|}
\Big\{\int_0^{2^kr_B}\big|t\partial_t e^{-t\sqrt \L}  (I-e^{-r_B\sqrt \L} )
g(x)\chi_{T(2^{k+1}B)\backslash T(2^{k}B) }\big|^2
{dt\over t}\Big\}^{1/2}\Big{\|}_{{ L}^{2}(2^kB)}.
\end{eqnarray*}
To estimate ${\rm B}_k,$ we set
$$
\Psi_{t,s} (\L)h(y)=(t+s)^2\Big({d^2{e^{-r\sqrt \L}}\over dr^2}
\Big|_{r=t+s}  h\Big)(y).
$$
Note that
$$
 ({ {I}}-e^{-r_B\sqrt \L} )g=\int_0^{r_B}s\sqrt \L e^{-s\sqrt \L}g{\frac{ds}{s}}.
$$
Let $\epsilon\in (0, 1/4)$. By Lemma~\ref{le2.4}, we have
\begin{align*}
 {\rm B}_k
& \leq C\Big{\|}\Big\{\int_0^{2^kr_B} \Big| \int_0^{r_B}
{ts\over (t+s)^2} \Psi_{t,s} (\L)g(x)\chi_{T(2^{k+1}B)\backslash T(2^{k}B) }{ds
\over s} \Big|^2  {dt \over t} \Big\}^{1/2}\Big{\|}_{{ L}^{2}(2^kB)}\\
&\leq C\Big{\|}\Big\{\int_0^{2^kr_B} \Big| \int_0^{r_B} \int_{B(x_B,r_B)}
{ts\over (t+s)^2}
{(t+s)^{\epsilon}\over (t+s+|x-y|)^{n+\epsilon}}\\&\quad \quad \quad \quad \quad \quad \quad \quad
\quad \quad \quad \quad \quad \quad \quad  \times |g(y)|\chi_{T(2^{k+1}B)\backslash T(2^{k}B) }(x)
{dyds\over s} \Big|^2  {dt \over t} \Big\}^{1/2}\Big{\|}_{{ L}^{2}(2^kB)}.
\end{align*}
Note that for   $x\in T(2^{k+1}B)\backslash T(2^{k}B)$
and $y\in B$, we have that $|x-y|\geq 2^kr_B$.   The inequality
$$
{ts(t+s)^{\epsilon}\over (t+s)^2}
\leq C\ \! {\rm min} \big(
(ts)^{\epsilon/2}, t^{-\epsilon/2}s^{3\epsilon/2}\big),
$$
together with H\"older's inequality and elementary integration, produces a positive constant
$C$  such that
\begin{align*}
{\rm B}_k &\leq  C  (2^k r_B)^{-(n+\epsilon) +{n\over 2}} \|g\|_{L^1(B)} \Big\{\int_0^{2^kr_B} \Big| \int_0^{r_B}
{\rm min} \big(
(ts)^{\epsilon/2}, t^{-\epsilon/2}s^{3\epsilon/2}\big)
{ds\over s} \Big|^2  {dt \over t} \Big\}^{1/2}\\
&\leq C (2^k r_B)^{-(n+\epsilon) +{n\over 2}} r_B^{{n\over 2} + \epsilon}\|g\|_{{ L}^2(B)},
\end{align*}
where we used the fact that
\begin{multline*}
  \Big\{\int_0^{2^kr_B} \big| \int_0^{r_B}
{\rm min} \big(
(ts)^{\epsilon/2}, t^{-\epsilon/2}s^{3\epsilon/2}\big)
{ds\over s} \big|^2  {dt \over t} \Big\}^{1/2}\\
\leq  \Big\{\int_0^{r_B} \big| \int_0^{r_B}
(ts)^{\epsilon/2}
{ds\over s} \big|^2  {dt \over t} \Big\}^{1/2}
 +   \Big\{\int_{r_B}^{\infty} \big| \int_0^{r_B}
  t^{-\epsilon/2}s^{3\epsilon/2}
{ds\over s} \big|^2  {dt \over t} \Big\}^{1/2}
\leq C    r_B^{ \epsilon}.
\end{multline*}
Consequently,
\begin{equation*}
{\rm A}_k
\leq C 2^{-k\epsilon}   r_B^{n\over 2}   ||| \mu_{\nabla_t, f}|||_{car}\norm{g}_{L^2(B)},
\end{equation*}
which implies
\begin{align*}
\int_{{\mathbb R}^{n+1}_+}
|F(x,t) G(x,t)|{dxdt\over t}
&\leq Cr_B^{{n  \over 2}}||| \mu_{\nabla_t, f}|||_{car}\|g\|_{{ L}^2(B)}
 +
C\sum_{k=2}^{\infty} 2^{-k\epsilon}  r_B^{{n  \over 2}}
|||\mu_f|||_{car}\|g\|_{{L}^2(B)} \\
&\leq  Cr_B^{{n  \over 2}}
||| \mu_{\nabla_t, f}|||_{car}\|g\|_{{L}^2(B)}
\end{align*}
as desired.
\end{proof}

\begin{lem} \label{le3.7}   Suppose $B, f, g, F, G $ are
defined as in Lemma~\ref{le3.6}. If  $||| \mu_{\nabla_t, f}|||_{car}<\infty$, then
we have the equality:
\begin{eqnarray*}
\int_{{\mathbb R}^n} f(x) ({\mathcal {I}}-e^{-r_B\sqrt{\L}})g(x)dx={1\over 4}\int_{{\mathbb R}^{n+1}_+}
F(x,t) G(x,t){dxdt\over t}.
\end{eqnarray*}
\end{lem}

\begin{proof}

 The proof  follows  by making minor modifications to that of  \cite[ Proposition 4]{DXY}, and so we skip it here.
See also \cite{DY2, DGMTZ, MSTZ}.
\end{proof}

\begin{proof}[Proof of Lemma~\ref{le3.5}]
First, we have an equivalent  characterization of $ {\rm BMO}_{\L}(\RR)$ that
   $f\in {\rm BMO}_{\L}(\RR) $ if and only if  $f\in L^2((1+|x|)^{-(n+\epsilon)}dx)$ and
  \begin{equation} \label{e3.9}
  \sup_B \Big( |B|^{-1}\int_{B}|f(x)- e^{-r_B\sqrt{\L}}f(x) |^2dx\Big)^{1/2}\leq C<\infty.
  \end{equation}
This can be obtained by  making minor modifications to that of  \cite[Proposition 6.11]{DY2}
  (see  also \cite{DY1, HLMMY}) corresponding to the case in which
the function $e^{-r_B\sqrt{\L}}f$ is replaced by $e^{-r_B^2\L}f$, and we omit the detail here.

 Now if $\|u\|_{{\rm HMO_\L}(\Real_+^{n+1})}<\infty$, then it follows from Lemma \ref{le3.1} that
 $$
 \int_{\Real^{n}} {|f_k(x)|^2\over 1+|x|^{2n}} dx\leq C_k <\infty.
 $$
 Given an $L^2$ function $g$ supported on a ball $B=B(x_B, r_B)$, it follows by  Lemma~\ref{le3.7}  that
 we have
\begin{align*}
  \int_{{\mathbb R}^n}
f_k(x) (I-e^{-r_B\sqrt \L} )g(x)dx
 = {1\over 4}\int_{{\mathbb R}^{n+1}_+}   t\partial_t e^{-t\sqrt \L}   f_k(x)  \
 t\partial_t e^{-t\sqrt \L}  (I-e^{-r_B\sqrt \L} )g(x) {dxdt\over t}.
\end{align*}
By Lemmas~\ref{le3.6} and ~\ref{le3.4},
\begin{align*}
  |\int_{{\mathbb R}^n}
f_k(x) (I-e^{-r_B\sqrt \L} )g(x)dx|  &\leq C|B|^{1/2}
 ||| \mu_{\nabla_t,  f_k}|||_{car} \|g\|_{{L}^2(B)}\\
 &\leq C|B|^{1/2}\|u\|_{{\rm HMO_\L}(\Real_+^{n+1})} \|g\|_{{L}^2(B)}.
\end{align*}
  Then the duality argument
for ${L}^2$ shows that
\begin{align*}
\Big(|B|^{-1}\int_B |f_k(x)-e^{-r_B\sqrt \L}f_k(x)|^2dx\Big)^{1/2}
&=|B|^{-1/2}\sup\limits_{\|g\|_{{ L}^2(B)\leq 1}}\Big|\int_{{\mathbb R}^n}
(I-e^{-r_B\sqrt \L})f_k(x)g(x)dx\Big|\nonumber\\
&=|B|^{-1/2}\sup\limits_{\|g\|_{{ L}^2(B)\leq 1}}\Big|\int_{{\mathbb R}^n}
f_k(x)\big(I-e^{-r_B\sqrt \L}\big)g(x)dx\Big|\nonumber\\
&\leq C
\|u\|_{{\rm HMO_\L}(\Real_+^{n+1})}
\end{align*}
for some $C>0$   independent of $k.$

 It then follows
 that for all $k\in{\mathbb N}$, $f_k$ is uniformly bounded in
${\rm BMO}_{\rm \L}({\mathbb R}^n)$.
\end{proof}


 \begin{proof}[Proof of part (1) of Theorem~\ref{th1.1}]  Letting $f_k(x)=u(x, 1/k)$,
 it follows by Lemma~\ref{le3.2} that $u(x, t+{1/k} )=    e^{-t\sqrt{\L}}(f_k)(x)$ and so
 \begin{align}\label{e3.10}
    \L u(x, t+{1/k})=  \L e^{-t\sqrt{\L}}(f_k)(x).
\end{align}
Then we have the following facts:

 \begin{itemize}

\item[(i)]    $H^1_\L(\Real^{n})$ is a Banach space;

\item[(ii)]
 For each $t>0$ and $x\in\RR$, $\partial^2_t {\mathcal P}_t(x, \cdot)\in H^1_{\L}(\RR)$ with $\|\partial^2_t {\mathcal P}_t(x, \cdot)\|_{H^1_{\L}(\RR)}
 \leq C/t^2$ (see \eqref{e2.8});

  \smallskip

\item[(iii)]
 The duality theorem that
 $ (H^1_\L(\Real^{n}) )^{\ast}= {\rm BMO_\L}(\Real^{n})$(see Lemma~\ref{le2.5}).

 \end{itemize}

 \noindent
  From (i), we use Lemma~\ref{le3.5} and pass  to a subsequence, we have that $f_k\rightarrow f$
  (in weak-${\ast}$ convergence) for some $f\in {\rm BMO}_\L(\Real^{n})$
 such that $\|f\|_{{\rm BMO}_\L(\Real^{n})}\leq C\|u\|_{{\rm HMO_\L}(\Real_+^{n+1})}$. Then by (ii) and (iii),
 we conclude that, for each $(x,t)\in {\mathbb R}_+^{n+1}$, the right-hand side
 of \eqref{e3.10} converges to $\L e^{-t\sqrt{\L}}(f)(x)$ when $k\to \infty$. On the other hand,
  as $k\rightarrow \infty$, the left-hand side of \eqref{e3.10} converges pointwisely
 to $ \L u(x, t)$.   Hence, for a fixed $t>0,$
 $$\L w(x, t) =0\ \ {\rm in}\ \
 {\mathbb R}^{n},
 $$
 where $w(x,t)=u(x, t)- e^{-t\sqrt{\L}}(f)(x)$.
Next, let us verify \eqref{e2.9}. Letting $k$ small enough such that $t>2/k$,   one writes
\begin{eqnarray*}
\int_{\RR} {|w(x,t)|\over 1+ |x|^{ 2n}}dx
& \leq&
   \int_{\RR} {|u(x, 1/k)|\over 1+ |x|^{ 2n}}dx \\
&& +\int_{\RR} {|u(x, t)-u(x, 1/k)|\over 1+ |x|^{ 2n}}dx\\
&&+ \int_{\RR} {|e^{-t\sqrt{\L}}(f)(x)|\over 1+ |x|^{ 2n}}dx
 \leq   I+II+III.
\end{eqnarray*}
By Lemma~\ref{le3.1}, we have that $I\leq C_k$.   For term $II$, we use  \eqref{e3.2} to obtain that $  | \partial_t u(x, t) |
  \leq   C /t,
$ and then
 \begin{eqnarray*}
 | u(x, t)-u(x, 1/k) |  =  \Big| \int_{1/k}^{t}   \partial_s u(x, s) ds \ \Big| \leq C\log(kt),
 \end{eqnarray*}
which gives
 \begin{eqnarray*}
II
& \leq& C \int_{\RR} {\log(kt)\over 1+ |x|^{2n} }dx
\leq C\log(kt).
\end{eqnarray*}

Consider the term $III$. For a fixed $t>0$ and $x\in R^n$, we set $f_B=t^{-n}\int_{B(x,t)} f(y)dy.$
It can be verified  by  a standard argument (see \cite[Theorem 2]{DGMTZ}) that
$|e^{-t\sqrt{\L}}f(x)|
 \leq C\|f\|_{{\rm BMO}_\L(\Real^{n})} +|f_B|.$    By Lemma 2 of  \cite{DGMTZ}, we have that
 $|f_B|\leq C(1+{\rm log}[{\rho(x)/t}])\|f\|_{{\rm BMO}_\L(\Real^{n})}$ if $t<\rho(x)$;
  $|f_B|\leq C \|f\|_{{\rm BMO}_\L(\Real^{n})}$ if $t\geq \rho(x)$. It then follows by Lemma~\ref{le2.1}  that
  there is a constant $k_0\geq 1$ such that
 \begin{eqnarray*}
|e^{-t\sqrt{\L}}f(x)| \leq  C\|f\|_{{\rm BMO}_\L(\Real^{n})}
\left(1+  {\rm log}\left[1 + {C\rho(0)\over t}\left(1+{|x|\over \rho(0)}\right)^{k_0\over k_0+1}\right] \right)
\end{eqnarray*}
 for every $t>0$ and   $x\in\RR,$
which yields
 $$
 III\leq  C\|f\|_{{\rm BMO}_\L(\Real^{n})} \int_{\RR} {1\over 1+ |x|^{2 n}}
 \left(1+  {\rm log}\left[ 1 + {C\rho(0)\over t}\left(1+{|x|\over \rho(0)}\right)^{k_0\over k_0+1}\right] \right)dx  \leq C_{t}<\infty.
 $$
  Estimate \eqref{e2.9} then follows readily.

By Lemma~\ref{le2.7}, we have that
 $u(x,t)=e^{-t\sqrt{\L}}(f)(x)$ with   $f\in {\rm BMO}_\L(\Real^{n}).$  The proof of part (1) of Theorem~\ref{th1.1}
 is complete.
\end{proof}

 \bigskip

\subsection{The characterization of ${\rm BMO}_{\L}(\RR)$ in terms of Carleson measure}

To prove part (2) of Theorem~\ref{th1.1}, we need  the following Lemmas~\ref{le3.8} and \ref{le3.9}.

\begin{lem}\label{le3.8}
 Suppose $V\in B_q$ for some $q> n.$  Let $  \beta=1-{n\over q}$.
  For every $N>0$, there exist    constants $C=C_{N}>0$ and $c>0$ such that
  for all $x,y\in\RR$ and $t>0,$ the semigroup kernels ${\mathcal K}_t(x,y)$,   associated to $e^{-t{\L}}$,
   satisfy the following estimates:
\begin{eqnarray}\label{e3.11}
 | \nabla_x {\mathcal K}_t(x,y)| + | \sqrt{t} \nabla_x \partial_t{\mathcal K}_t(x,y)|
 \leq C t^{-(n+1)/2}e^{-\frac{\abs{x-y}^2}{ct}}\left(1+\frac{\sqrt{t}}{\rho(x)}
+\frac{\sqrt{t}}{\rho(y)}\right)^{-N},
\end{eqnarray}
and for $|h|<|x-y|/4,$
\begin{eqnarray}\label{e3.12}
 | \nabla_x{\mathcal K}_t(x+h,y)- \nabla_x{\mathcal K}_t(x,y)|
 \leq C\Big({|h|\over \sqrt{t}}\Big)^{\beta}
t^{-(n+1)/2}e^{-\frac{\abs{x-y}^2}{ct}}.
\end{eqnarray}
\end{lem}

\begin{proof} The proof of Lemma~\ref{le3.8} is inspired by ideas developed in \cite{Shen}.
Let $\Gamma_0(x,y)$ denote the fundamental
solution for the operator $-\Delta$ in $\Real^n.$ It is well-known that
$$
\Gamma_0(x,y)=  -{1\over n(n-2)\omega(n)} {1\over |x-y|^{n-2}}, \ \ \ \ n\geq 3,
$$
where $\omega(n)$ is the area of the unit sphere in $\Real^n$.
Fix $t>0$ and  $x_0, y_0\in \Real^{n}$. Assume that  $\partial_tu+\L u=0$.
Let $\eta\in C_0^{\infty}(B(x_0, 2R))$ such that
$\eta=1$ on $B(x_0, 3R/2)$,  $|\nabla\eta|\leq C/R$ and $|\nabla^2 \eta|\leq C/R^2.$
Following an argument as in Lemma 4.6 of \cite{Shen}, we have
\begin{align}\label{e3.13}
u(x)\eta(x)&=\int_{\Real^n}  \Gamma_0(x,y)\{ -\Delta \} (u\eta)(y) dy\nonumber\\
&=\int_{\Real^n}  \Gamma_0(x,y)\{ -V u \eta - \eta\partial_t u -2\nabla u \cdot \nabla \eta -u\Delta \eta\}   dy\nonumber\\
&=\int_{\Real^n}  \Gamma_0(x,y)\{ -V u \eta -\eta\partial_t u -u\Delta \eta  \}   dy\nonumber\\
&\quad + 2\int_{\Real^n}  \nabla\Gamma_0(x,y) \cdot (\nabla \eta) u    dy.
\end{align}
Thus, for $x\in B(x_0, R),$
\begin{align*}
|\nabla_x u(x)| &\leq C\int_{B(x_0,2R)}  {V(y)|u(y)| |\eta(y)| \over |x-y|^{n-1}}  dy \\
&\hspace{ 2cm} +
C\int_{B(x_0,2R)}  { |\partial_t u(y) | |\eta(y)| \over |x-y|^{n-1}}  dy+{C\over R^{n+1}} \int_{B(x_0,2R)}   |u(y)|  dy.
\end{align*}
It follows from Lemmas 1.2 and 1.8 in \cite{Shen} that there exist $C$ and $m_0>0$ such that
for all $R>0,$
$$
\int_{B(x_0,2R)}  {V(y)  \over |x-y|^{n-1}}  dy  \leq {C\over R^{n-1}}\int_{B(x_0,2R)}   V(y)   dy
\leq {C\over R }   \Big({R\over \rho(x_0)}\Big)^{m_0}.
 $$
Therefore,
 \begin{align*}
|\nabla_x u(x)| &\leq {C\over R}
\sup_{B(x_0,2R)}|u(y)| \Big\{ \Big({R\over \rho(x_0)}\Big)^{m_0}  + 1\Big\}
 +  CR\sup_{B(x_0,2R)}|\partial_t u(y) |,\ \ \ \  \forall x\in B(x_0, R).
\end{align*}

Let us prove \eqref{e3.11}.
We take  $u={\mathcal K}_t(x, y_0)$ and  $R=|x_0-y_0|/8$. If $x\in B(x_0, 2R)$, then $\abs{x-y_0}\sim \abs{x_0-y_0}$.
Using Lemma \ref{le2.2}, we have that for any $N>m_0+1,$
\begin{align}\label{e3.14}
|\nabla_x {\mathcal K}_t(x_0, y_0) | &\leq {C\over R}
\sup_{B(x_0,2R)}|{\mathcal K}_t(x, y_0)| \Big\{ \Big({R\over \rho(x_0)}\Big)^{m_0}  + 1\Big\}
 +  R\sup_{B(x_0,2R)}|\partial_t {\mathcal K}_t(x, y_0) |\nonumber\\
&\leq  {C_N\over R} t^{-n/2}e^{-\frac{\abs{x_0-y_0}^2}{ct}}\Big( 1
+\frac{\sqrt{t}}{\rho(x_0)}+\frac{\sqrt{t}}{\rho(y_0)}\Big)^{-N}
\Big\{ \Big({R\over \rho(x_0)}\Big)^{m_0}  + 1 +{R^2\over t}\Big\}\nonumber\\
&\leq  C'_N  t^{-(n+1)/2}e^{-\frac{\abs{x_0-y_0}^2}{c_1t}}\left( 1
+\frac{\sqrt{t}}{\rho(x_0)}+\frac{\sqrt{t}}{\rho(y_0)}\right)^{-N+m_0},
\end{align}
which implies the desired result.
Now, letting $u=\sqrt{t}\partial_t{\mathcal K}_t(x, y_0)$,
we use a similar argument as in \eqref{e3.14} to obtain   estimate  for
the term $|\sqrt{t}\nabla_x\partial_t{\mathcal K}_t(x_0, y_0)|$ in
 \eqref{e3.11}.

To prove \eqref{e3.12}, we follow  an argument as in Remark 4.10 and (6.6) of \cite{Shen}.
It follows from   \eqref{e3.13}  that
\begin{align*}
\|\nabla^2_x(u\eta)\|_{q} \leq C R^{(n/q)-2}
\Big\{ \Big({R\over \rho(x_0)}\Big)^{m_0}  + 1\Big\} \sup_{B(x_0,2R)}|u|  +  CR^{(n/q)}\sup_{B(x_0,2R)}|\partial_t u |.
\end{align*}
To see \eqref{e3.12}, we fix $x_0, y_0\in \RR$ and $h\in\RR, |h|<|x_0-y_0|/4$.
Let  $u= {\mathcal K}_t(x, y_0)$ and  $R=|x_0-y_0|/8$.
It then follows from the imbedding theorem of  Morrey that
\begin{eqnarray*}
&&\hspace{-1.5cm}|\nabla_x{\mathcal K}_t(x_0+h, y_0)-\nabla_x{\mathcal K}_t(x_0, y_0)|
\leq  C|h|^{1-(n/q)} \Big( \int_{B(x_0, R)} |\nabla^2_x {\mathcal K}_t(x, y_0)|^{q}dx\Big)^{1/q}\\
&\leq &C \Big({|h|\over R}\Big)^{1-(n/q)} {1\over R}
\Big\{ \Big({R\over \rho(x_0)}\Big)^m  + 1\Big\} \left( \sup_{B(x_0,2R)}| {\mathcal K}_t(x, y_0)|
 +  {R^2\over t }\sup_{B(x_0,2R)}|t\partial_t  {\mathcal K}_t (x, y_0) |\right)\\
&\leq &C \Big({|h|\over R}\Big)^{1-(n/q)} t^{-(n+1)/2}e^{-\frac{\abs{x_0-y_0}^2}{ct}}\\
&\leq &C \Big({|h|\over \sqrt{t}}\Big)^{1-(n/q)} t^{-(n+1)/2}e^{-\frac{\abs{x_0-y_0}^2}{c't}}.
\end{eqnarray*}
Estimate  \eqref{e3.12} follows readily.
\end{proof}

 Turning to the Poisson semigroup $\{e^{-t\sqrt{\L}}\}_{t>0}$, we have

\begin{lem}\label{le3.9}
 Suppose $V\in B_q$ for some $q> n.$ Let $\beta=1-{n\over q}$.  For every $N>0$, there exist    constants $C=C_{N}>0$
 such that for all $x,y\in\RR$ and $t>0$,

 \begin{itemize}

\item[(i)]
 ${\displaystyle
|t\nabla_x{\mathcal P}_t(x,y)|
            \leq C \frac{t}{(t^2+ \abs{x-y}^2)^{\frac{n+1}{2}}}
              \left(1+\frac{(t^2+\abs{x-y}^2)^{1/2}}{\rho(x)}
             + \frac{(  t^2 +\abs{x-y}^2)^{1/2}}{\rho(y)}\right)^{-N}; }
 $

\item[(ii)] For all $|h|\leq |x-y|/4$,
\begin{eqnarray*}
 |t\nabla_x{\mathcal P}_t(x+h,y)-t\nabla_x{\mathcal P}_t(x,y)|
 \leq C\left({|h|\over t}\right)^{\beta}
{t\over (t^2+|x-y|^2)^{{n+1\over 2}}};
\end{eqnarray*}

\item[(iii)]  There is some $\delta>1$ such that
${\displaystyle  \big| t\nabla_x e^{-t\sqrt{\L}}(1)(x) \big|\le C \min \Big\{
  \Big(\frac{ {t}}{\rho(x)}\Big)^{\delta}, \Big( {t\over \rho(x)}\Big)^{-N}\Big\}.}$
 \end{itemize}
\end{lem}

\begin{proof}
The proofs of (i) and (ii) follow  from
the subordination formula \eqref{e2.6} and Lemma~\ref{le3.8}.
 To  prove (iii), we consider two cases.

\smallskip

\noindent
{\it Case 1: }  $t> \rho(x).$ We use  (i) to obtain
\begin{align*}
 | t\nabla_x e^{-t\sqrt{\L}}(1)(x)  |
 \leq C  \left( {t\over \rho(x) }\right)^{-N}.
\end{align*}

\smallskip

\noindent
{\it Case 2: }  $t\leq \rho(x).$ In this case, it follows from  the subordination formula \eqref{e2.6} that
\begin{align*}
\big| t\nabla_x e^{-t\sqrt{\L}}(1)(x) \big|
&\le  C\int_0^ {\infty} \Big({t\over u}\Big)^2 \exp \big(-{t^2\over 4u}\big)    \left|\sqrt{u} \nabla_x e^{-u {\L}}(1)(x)\right|    ~du\\
 &\leq   C \int_0^{\rho(x)^2} \Big({t\over u}\Big)^2 \exp \big(-{t^2\over 4u}\big)    \left|\sqrt{u} \nabla_x e^{-u {\L}}(1)(x)\right|    ~du \\
 &\ + \ C\int_{\rho(x)^2}^{\infty} \Big({t\over u}\Big)^2 \exp \big(-{t^2\over 4u}\big)    \left|\sqrt{u} \nabla_x e^{-u {\L}}(1)(x)\right|    ~du \\
 &=I_1(x) +I_2(x).
\end{align*}
By Lemma~\ref{le3.8}, we have
\begin{align*}
I_2(x)
 \le C\int_{\rho(x)^2}^{\infty} \Big({t\over u}\Big)^2
  ~du \le C\left(\frac{t}{\rho(x)}\right)^{2}.
\end{align*}

Consider  the term $I_1(x).$ We apply Kato-Trotter formula  to obtain
\begin{align*}\label{eeeee}
h_u(x-y)- {\mathcal K}_u(x,y)
=  \int_0^u\int_{\Real^n} h_s(x-z)V(z){\mathcal K}_{u-s}(z, y)dzds,
\end{align*}
which gives
\begin{align*}
 |{\sqrt{u}}\nabla_x e^{-u\L}(1)(x)  \big|
 &= \Big|\int_{\Real^n}\int_0^u\int_{\Real^n}  \sqrt{u}\nabla_xh_s(x-z)V(z){\mathcal K}_{u-s}(z, y)dzdsdy\Big| \\
&\le C  \int_0^u  \Big({u\over s}\Big)^{1/2}\int_{\Real^n} s^{-n/2} \exp\Big(-{|x-z|^2\over cs}\Big)V(z) dzds.
\end{align*}
It follows from  \eqref{e2.2}
that  for $s\le \rho(x)^2,$
\begin{eqnarray*}
  \int_{\Real^n}s^{-n/2} \exp\Big(-{|x-y|^2\over cs}\Big) V(y)  dy\le
\frac C{s}\left(\frac{\sqrt s}{\rho(x)}\right)^{\delta},
\end{eqnarray*}
where $\delta=2-\frac{n}{q}>1$. This implies that
\begin{align*}
 |{\sqrt{u}}\nabla_x e^{-u\L}(1)(x)  |
&\le C  \int_0^u  \Big({u\over s}\Big)^{1/2}  s^{-1} \Big( {\sqrt{s} \over \rho(x)}\Big)^{\delta}  ds\le C\Big( {\sqrt{u} \over \rho(x)}\Big)^{\delta}.
\end{align*}
Therefore,
\begin{align*}
 I_1(x)
&\le  C\int_0^ {\rho(x)^2} \Big({t\over u}\Big)^2 \exp \big(-{t^2\over 4u}\big)   \left(\frac{\sqrt{u}}{\rho(x)}\right)^{\delta}
 ~du\le C\left(\frac{t}{\rho(x)}\right)^{\delta}.
\end{align*}
Combining the estimates of {\it Case 1} and {\it Case 2}, we obtain (iii).
\end{proof}

 \medskip

To prove part (2) of Theorem~\ref{th1.1}, we  recall that the  Carleson
measure is closely related to  the   space ${\rm BMO}_\L(\RR)$.  We  note that
 for every $f\in {\rm BMO}_\L(\RR)$,
\begin{equation}\label{e3.18}
 \mu_{\nabla_t, f}(x,t)=t| \partial_te^{-t\sqrt{\L}}(f)(x)|^2 {dx dt}
\end{equation}
is a Carleson measure with $||| \mu_{\nabla_t, f_k}|||_{car} \leq C\|f\|^2_{{\rm BMO_\L}(\RR)}$(see \cite{DY2, MSTZ}).

\begin{proof}[Proof of part (2) of Theorem~\ref{th1.1}]
 Recall that the condition $V\in B_n$   implies $V\in B_{q_0}$ for some $q_0>n$.
From
Lemmas~\ref{le2.4} and  \ref{le3.9}, we see  that $u(x,t)=e^{-t\sqrt{\L}}f(x)\in C^1({\mathbb R}^{n+1}_+)$.
Let us fix a ball $B=B(x_B, r_B)$. From \eqref{e3.18},
it suffices to  show that there exists a constant $C>0$ such that
\begin{align*}
 \int_0^{r_B}\int_B |t\nabla_{x}e^{-t\sqrt { \L}}f(x)|^2\frac{dxdt}{t}
&\le C|B|\norm{f}_{{\rm BMO_\L}(\RR)}^2.
\end{align*}
To do this, we split the function $f$ into local, global, and constant parts as follows
$$f=(f-f_{2B})\chi_{2B}+(f-f_{2B})\chi_{(2B)^c}+f_{2B}=f_1+f_2+f_3,
$$
where $2B=B(x_B, 2r_B)$.

Since the Riesz transform $\nabla \L^{-1/2}$ is bounded on $L^2(\RR)$,  we have
\begin{align*}
 \int_0^{r_B}\int_B |t\nabla_{x}e^{-t\sqrt { \L}}f_1(x) |^2\frac{dxdt}{t}
 &= \int_0^{r_B}\int_{\Real^n} |\nabla_{x}\L^{-{1/2}}t\L^{{1/2}}e^{-t\sqrt {\L}}f_1(x) |^2\frac{dxdt}{t}\\
&\le C \int_0^{\infty}\int_{\RR} |t\L^{1/2}e^{-t\sqrt { \L}}f_1(x) |^2\frac{dxdt}{t}\\
&\le C\norm{f_1}_{L^2(\Real^n)}^2\\
&=C\int_{2B}\abs{f(x)-f_{2B}}^2dx\\
&\le C|B|\norm{f}_{{\rm BMO_\L}(\RR)}^2.
\end{align*}

To estimate the global term, we use (i) of Lemma~\ref{le3.9}  and then the standard argument as in Theorem 2 of \cite{DGMTZ}   shows that
for $x\in B$ and $t<r_B$,
\begin{align*}
&| t\nabla_x e^{-t\sqrt \L}f_2(x) |
 \le C\int_{(2B)^c} \abs{f(y)-f_{2B}}\frac{t}{\abs{x_B-y}^{n+1}}dy\\
&\le C\sum_{k=2}^\infty \frac{t}{(2^kr_B)^{n+1}}\left[\int_{2^kB\backslash 2^{k-1}B}
\abs{f(y)-f_{2^kB}}dy+(2^kr_B)^n\abs{f_{2^kB}-f_{2B}}\right]\\
&\le C\big({t\over r_B}\big)\sum_{k=2}^\infty 2^{-k}\left[\norm{f}_{{\rm BMO_\L}(\RR)}
+k\norm{f}_{{\rm BMO_\L}(\RR)}\right]\le C\big({t\over r_B}\big)\norm{f}_{{\rm BMO_\L}(\RR)},
\end{align*}
which yields
\begin{align*}
 \int_0^{r_B}\int_B |t\nabla_{x}e^{-t\sqrt { \L}}f_2(x) |^2\frac{dxdt}{t}\le C|B| r_B^{-2}\int_0^{r_B}
t {dt} \norm{f}^2_{{\rm BMO_\L}(\RR)}\leq  C|B|\norm{f}^2_{{\rm BMO_\L}(\RR)}.
\end{align*}
It remains to estimate the constant term $f_3=f_{2B}$, for which we make use of (i) and (iii) of Lemma~\ref{le3.9}.
  Assume first that $r_B\le \rho(x_B)$. By Lemma~\ref{le3.1}, $\rho(x)
\sim \rho(x_B)$ for $x\in B$, we have
 \begin{eqnarray}\label{e3.19}
 \int_0^{r_B}\int_B  |t\nabla_{x}e^{-t\sqrt{\L}}f_3(x)|^2\frac{dxdt}{t}
&\leq&
 {\abs{f_{2B}}^2}  \int_0^{r_B}\int_B  \big(t/\rho(x)\big)^{2\delta}  \frac{dxdt}{t}\nonumber\\
&\leq&
 C|B|\abs{f_{2B}}^2   \big(r_B/\rho(x_B)\big)^{2\delta}\nonumber\\
 &\leq&
 C|B|\|f\|_{{\rm BMO_\L}(\RR)}^2\Big(1+\log {\rho(x_B)\over r_B}\Big)^2   \big(r_B/\rho(x_B)\big)^{2\delta}\nonumber\\
  &\leq&
 C|B|\|f\|_{{\rm BMO_\L}(\RR)}^2.
\end{eqnarray}

Suppose finally that $r_B> \rho(x_B)$,  we use an argument as in Theorem 2 of \cite{DGMTZ} to select a finite family
of critical balls $\set{Q_k}$ such that $B\subset \cup Q_k$ and $\sum\abs{Q_k}\le \abs{B}$.
 Then, using the fact that  $\abs{f_{2B}}\le \norm{f}_{\rm BMO_\L(\RR)},$
 we can bound the left hand side of \eqref{e3.19} by
\begin{align*}
&C{\norm{f}_{\rm BMO_\L}^2 }\sum_k\left(\int_0^{\rho(x_k)}\int_{Q_k}
\left({t\over {\rho(x_k)}}\right)^{2\delta}{dxdt\over t}+\int_{\rho(x_k)}^\infty
\int_{Q_k}\left(\frac{t}{\rho(x_k)}\right)^{-2N}{dxdt\over t}\right)\\
&\le  C{\norm{f}_{{\rm BMO_\L}(\RR)}^2 }\sum_k \abs{Q_k}\le C|B|\norm{f}_{{\rm BMO_\L}(\RR)}^2,
\end{align*}
which establishes the proof of part (2) of Theorem 1.1.
 \end{proof}

\bigskip

\medskip

\section{The spaces ${\rm HMO}^\alpha_{\L}({\mathbb R}^{n+1}_+)$ and their characterizations}
\setcounter{equation}{0}

In this section we will extend the method   for the space ${\rm BMO}_{\L}(\RR)$ in Section 3
 to obtain some generalizations to  Lipschitz-type spaces $\Lambda^{\alpha}_{\L}(\RR)$ with
  $0<\alpha<1$ (see \cite{BHS2008}).

 Let us recall that a locally integrable function $f$ in
 $\Lambda^{\alpha}_{\L}(\RR), 0<\alpha<1,$ if  there exists
 a constant $C>0$ such that
 \begin{eqnarray}\label{e4.1}
 |f(x)-f(y)|\leq C |x-y|^{\alpha}
\end{eqnarray}
 and
 \begin{eqnarray}\label{e4.2}
 |f(x)|\leq C \rho(x)^{\alpha}
\end{eqnarray}
 for all $x,y\in\RR.$  The norm of $f$ in  $\Lambda^{\alpha}_{\L}(\RR)$ is given by
 $$
 \|f\|_{\Lambda^{\alpha}_{\L}(\RR)}=\sup_{\substack{ x,y\in\RR\\
 x\not=y}}{|f(x)-f(y)|\over |x-y|^{\alpha}} +\sup_{x\in\RR} |\rho(x)^{-\alpha}f(x)|.
 $$
 Because of \eqref{e4.2}, this space $\Lambda^{\alpha}_{\L}(\RR)$  is in fact a proper subspace of  the classical Lipschitz space
 $\Lambda^{\alpha}(\RR)$ (see \cite{BHS2008, FJN, MSTZ}).
   Following \cite{BHS2008},  a locally integrable function $f$ in $\Real^n$ is in ${\rm BMO}_{\L}^{\alpha}(\RR)$, $\alpha> 0,$
  if there is a constant $C>0$ such that
\begin{eqnarray}\label{e4.3}
\int_{B}|f(y)-f_B|dy\leq Cr^{n+\alpha}
\end{eqnarray}
for every ball $B=B(x, r)$, and
\begin{eqnarray}\label{e4.4}
\int_{B} |f(y)|~dy\leq Cr^{n+\alpha}
	\end{eqnarray}
 for every  ball  $B=B(x, r)$  with $r\geq \rho(x)$. Define
 \begin{eqnarray*}
 \|f\|_{{\rm BMO}_{\L}^{\alpha}(\RR)} =\inf\big\{ C: C \ {\rm satisfies }\  \eqref{e4.3} \ {\rm and}\  \eqref{e4.4}  \big\}.
\end{eqnarray*}
It is proved in \cite[Proposition 4]{BHS2008}  that for $0<\alpha<1$,
$ {\rm BMO}_{\L}^{\alpha}(\RR)=\Lambda^{\alpha}_{\L}(\RR).$

\begin{thm}\label{th4.1}
Suppose $V\in B_q$ for some $q\geq n,$ and $\alpha\in (0, 1).$ We denote by ${\rm HMO^{\alpha}_\L}(\Real_+^{n+1})$  the class of all $C^1$-functions $u(x,t)$
of the solution of ${\mathbb L}u=-u_{tt}+\L u=0$ such that
\begin{eqnarray}\label{e4.5}
\|u\|^2_{{\rm HMO^{\alpha}_\L}(\Real_+^{n+1})}&=& \sup_{x_B, r_B}     r_B^{-(n+2\alpha)} \int_0^{r_B}\int_{B(x_B, r_B)} t
| \nabla u(x,t) |^2
{dx dt }  <\infty.
\end{eqnarray}
Then we have
 \begin{itemize}
\item[(1)]  If $u\in {\rm HMO^{\alpha}_\L}(\Real_+^{n+1})$, then    there exist some $f\in \Lambda^{\alpha}_{\L}(\RR) $
and a constant $C>0$ such that $u(x,t)=e^{-t\sqrt \L}f(x)$,
and
 $
\|f\|_{\Lambda^{\alpha}_{\L}(\RR) }\leq C\|u\|_{{\rm HMO^{\alpha}_\L}(\Real_+^{n+1})}.
 $

\item[(2)]  If $f\in \Lambda^{\alpha}_{\L}(\RR) $, then  $u(x,t)=e^{-t\sqrt \L}f(x)\in {\rm HMO^{\alpha}_\L}(\Real_+^{n+1})$ with
 $
 \|u\|_{{\rm HMO^{\alpha}_\L}(\Real_+^{n+1})}\approx \|f\|_{\Lambda^{\alpha}_{\L}(\RR) }.
 $
 \end{itemize}
\end{thm}

Part (2) of Theorem~\ref{th4.1} is a straightforward result from  the following proposition.

\begin{prop}\label{prop4.2} Suppose $V\in B_q$ for some $q\geq n,$
and $\alpha\in(0, 1)$ and  $f$ is a function such that
 $$\int_{\Real^n}\frac{\abs{f(x)}}{(1+\abs{x})^{n+\alpha+\varepsilon}}~dx<\infty$$
 for some $\varepsilon>0$. Then the following statements are equivalent:

 \begin{itemize}
\item[(1)]   $f
\in \Lambda^{\alpha}_{\L}(\RR) $;

 \item[(2)]  There exists a constant $C>0$ such that
 \begin{equation*}
\|t\nabla e^{-t\sqrt {\L}}f(x) \|_{L^\infty(\Real^n)}\le Ct^\alpha;
 \end{equation*}

\item[(3)]   $u(x,t)=e^{-t\sqrt {\L}}f(x)\in {\rm HMO^{\alpha}_\L}(\Real_+^{n+1})$
 with
 $
 \|u\|_{{\rm HMO^{\alpha}_\L}(\Real_+^{n+1})}\approx \|f\|_{\Lambda^{\alpha}_{\L}(\RR) }.
 $
 \end{itemize}
\end{prop}
\begin{proof}
{$\rm (1)\Rightarrow (2)$.}
Let $f\in \Lambda^{\alpha}_{\L}(\RR)$. One writes
\begin{align*}
   t\nabla e^{-t\sqrt {\L}}f(x)  &=  \int_{\Real^n}t\nabla {\mathcal P}_t(x,z)\left(f(z)-f(x)\right)~dz
    +f(x)
	t\nabla e^{-t\sqrt {\L}}(1)(x)  \\
    & =  I(x)+II(x).
\end{align*}
From Lemma~\ref{le3.9}, we have
 \begin{eqnarray*}
 |I(x)|&\leq & C\|f\|_{\Lambda^{\alpha}_{\L}(\RR)} \int_{\Real^n}|t\nabla {\mathcal P}_t(x,z)|\abs{x-z}^\alpha
     dz\\
	 &\leq& C\|f\|_{\Lambda^{\alpha}_{\L}(\RR)}
	 \int_{\Real^n}\frac{t\abs{x-z}^\alpha}{\left(t+\abs{x-z}\right)^{n+1}}~dz\\
 &\leq& Ct^{\alpha}\|f\|_{\Lambda^{\alpha}_{\L}(\RR)}.
\end{eqnarray*}

To estimate the term $II(x)$, we  consider two
cases.

\smallskip

\noindent
{\it Case 1:}   $\rho(x)\leq t$. In this case we use Lemma~\ref{le3.9} to obtain
  \begin{eqnarray*}
 |II(x)|&\leq & \|f\|_{\Lambda^{\alpha}_{\L}(\RR)} \rho(x)^\alpha\big|t\nabla e^{-t\sqrt {\L}}(1)(x)\big|\\
 &\leq& Ct^{\alpha}\|f\|_{\Lambda^{\alpha}_{\L}(\RR)}.
\end{eqnarray*}

\smallskip

\noindent
{\it Case 2:}   $\rho(x)> t$.  By  Lemma~\ref{le3.9}, there exists a $\delta>1$ such that
 \begin{eqnarray*}
  |II(x)|&\leq&
C\|f\|_{\Lambda^{\alpha}_{\L}(\RR)}\rho(x)^\alpha \Big({t\over \rho(x)}\Big)^{\delta}\\
&\leq&
C\|f\|_{\Lambda^{\alpha}_{\L}(\RR)}\rho(x)^\alpha \Big({t\over \rho(x)}\Big)^{\alpha}\\
&=&
Ct^{\alpha}\|f\|_{\Lambda^{\alpha}_{\L}(\RR)},
\end{eqnarray*}
which, together with estimate of $I(x)$, yields
$ \|t\nabla e^{-t\sqrt {\L}}f(x) \|_{L^\infty(\Real^n)}\le Ct^{\alpha}\|f\|_{\Lambda^{\alpha}_{\L}(\RR)}.
$

\smallskip

{$\rm (2)\Rightarrow (3)$.}
For every ball $B=B(x_B,r_B)$,  we have
\begin{align*}
 \int_0^{r_B}\int_B  t| \nabla e^{-t\sqrt {\L}}f(x)|^2 {dx~dt}
 &\leq C\|f\|^2_{\Lambda^{\alpha}_{\L}(\RR)}
\int_0^{r_B}\int_B t^{2\alpha-1}~ {dt~dx}\\
&\leq Cr_B^{n+2\alpha}\|f\|^2_{\Lambda^{\alpha}_{\L}(\RR)},
\end{align*}
and hence   $u(x,t)=e^{-t\sqrt {\L}}f(x)\in {\rm HMO_\L^\alpha}(\Real_+^{n+1}).$

\smallskip

 The proof of {$\rm (3)\Rightarrow (1)$} is a direct consequence of    \cite[Theorem 1.3]{MSTZ}.
 This completes the proof.
 \end{proof}

To prove part (1) of Theorem~\ref{th4.1}, we need some preliminary results.

\begin{lem}\label{le4.3}
Let $\alpha\in (0, 1).$
 For every  $u\in {\rm HMO}^{\alpha}_{\L}(\Real_+^{n+1})$ and
  every $k\in{\mathbb N}$, there exists a constant $C_{k, \alpha}>0$ such that
\begin{equation*}
\int_{\RR}{|u(x,{1/k})|^2\over 1+|x|^{2n+1}}  dx\leq C_{k, \alpha} <\infty,
\end{equation*}
and so $u(x, 1/k)\in L^2((1+|x|)^{-(2n+1)}dx)$. Hence for  all $k\in{\mathbb N}$, $e^{-t\sqrt{\L}}(u(\cdot, {1/k}))(x)$ exists
everywhere in ${\mathbb R}^{n+1}_+$.
\end{lem}

\begin{lem} \label{le4.4} Let $\alpha\in (0, 1).$ For every  $u\in {\rm HMO}^{\alpha}_{\L}(\Real_+^{n+1})$ and for every $k\in{\mathbb N}$,
there exists a constant $C>0$ independent of $k$ such that
\begin{eqnarray*}
\|u(\cdot, 1/k)\|_{{\rm BMO}^{\alpha}_{\L}(\RR)}
&\leq& C\|u\|_{{\rm HMO}^{\alpha}_{\L}(\Real_+^{n+1})}.
\end{eqnarray*}
Hence for  all $k\in{\mathbb N}$, $u(x, 1/k)$ is uniformly bounded in  ${\rm BMO}^{\alpha}_{\L}(\RR)$.
\end{lem}

The proof of Lemma~\ref{le4.3} is   analogous to that  of Lemma~\ref{le3.1}. For Lemma~\ref{le4.4},
similar arguments as in Lemmas~\ref{le3.2} and \ref{le3.4} show  that
$u(x,t+1/k)=e^{-t\sqrt{\L}}\big( u(\cdot, 1/k)\big)(x)$ satisfies
\begin{eqnarray}\label{e4.6} \hspace{1cm}
\sup_{x_B, r_B}     r_B^{-(n+2\alpha)} \int_0^{r_B}\int_{B(x_B, r_B)} t | \partial_t e^{-t\sqrt{\L}}\big( u(\cdot, 1/k)\big)(x) |^2
{dx dt } \leq C\|u\|^2_{{\rm HMO}^{\alpha}_{\L}(\Real_+^{n+1})}
\end{eqnarray}
 for all $k\in{\mathbb N}.$
This,  together with  \cite[Theorem 1.3]{MSTZ}, shows that $u(x, 1/k)\in {\rm BMO}^{\alpha}_{\L}(\RR)$, and then
 the norm of $u(x, 1/k)$ in  the space ${\rm BMO}^{\alpha}_{\L}(\RR)$ is less than  the left hand side of \eqref{e4.6}.
 We omit the detail here.

 \smallskip

 \begin{proof}[Proof of part   (1)  Theorem~\ref{th4.1}] By  using Lemma~\ref{le2.6} for   $\partial_{t}u(x, t+{1/k})$ and
 a similar argument as in \eqref{e3.2}  we show that
  \begin{equation*}
 | \partial_{t}u(x, t) | \leq C t^{\alpha-1}  \|u\|_{{\rm HMO}^{\alpha}_{\L}(\Real_+^{n+1})},
  \end{equation*}
and hence if $t_1, t_2>0$,
 \begin{align*}
 |u(x, t_1)-u(x, t_2)|&=\big|\int_{t_2}^{t_1}\partial_s u(x, s)ds\big|
 \leq C \|u\|_{{\rm HMO}^{\alpha}_{\L}(\Real_+^{n+1})}
 \big|\int_{t_2}^{t_1}s^{\alpha-1}ds\big|\\
 &\leq C|t_1^{\alpha}-t_2^{\alpha}|
  \leq C|t_1-t_2|^{\alpha}
\end{align*}
since $0<\alpha<1.$ The family $\{u(x, t)\}$ is a Cauchy sequence as $t$ tends to zero and hence converges to some function $f(x)$
everywhere.

Now we apply Lemma~\ref{le4.4}, and note that for all $k\in{\mathbb N}$,
\begin{equation*}
\|u(\cdot, 1/k)\|_{{\rm BMO}^{\alpha}_{\L}(\RR)}\leq  C\|u\|_{{\rm HMO}^{\alpha}_{\L}(\Real_+^{n+1})}<\infty.
\end{equation*}
Letting $k$ tend to $\infty$,    we conclude
\begin{equation*}
\|f\|_{\Lambda^{\alpha}_{\L}(\RR)}\leq C
\|f\|_{{\rm BMO}^{\alpha}_{\L}(\RR)} \leq C\|u\|_{{\rm HMO}^{\alpha}_{\L}(\Real_+^{n+1})},
\end{equation*}
and hence
 $u(x,t)=e^{-t\sqrt{\L}}f(x).$
This completes the proof of part (1) of Theorem~\ref{th4.1}.
 \end{proof}

\bigskip

{\bf Acknowledgments.} X.T. Duong was
supported by Australia Research Council (ARC).  L.  Yan was supported by
 NNSF of China (Grant No.  10925106 and 11371378),
 Guangdong Province Key Laboratory of Computational Science
 and Grant for Senior Scholars from the Association of Colleges and Universities of Guangdong.
 C. Zhang was supported by the General Financial Grant from the China Postdoctoral Science Foundation (Grant No. 2013M531883).
 L.X. Yan would like to thank S. Hofmann, Z.W. Shen and J. Xiao for helpful discussions.

 \bigskip




\begin{thebibliography}{10}
 \bibitem{BHS2008} B. Bongioanni, E. Harboure and O. Salinas,
 {Weighted inequalities for negative powers of Schr\"odinger operators}.
 \textit{J. Math. Anal. Appl.}
 \textbf{348} (2008), 12--27.

  \bibitem{BHS2009} B. Bongioanni, E. Harboure and O. Salinas,
 {Riesz transforms related to  Schr\"odinger operators acting on BMO type spaces}.
 \textit{J. Math. Anal. Appl.} \textbf{357}(2009), 115--131.


 \bibitem{CMS} R.R. Coifman, Y. Meyer and E.M. Stein,
  Some new functions and their applications to harmonic analysis.
  {\it J. Funct. Analysis}, {\bf 62}(1985), 304--335.

 \bibitem {DKP} M. Dindos, C. Kenig and J.  Pipher,
  BMO solvability and the $A_{\infty}$ condition for elliptic operators. \textit{J. Geom. Anal. } \textbf{21} (2011), 78--95.





\bibitem   {DY1}  X.T. Duong and L.X. Yan, New function spaces of
{\rm BMO} type, the John-Nirenberg inequality, interpolation and applications,
{\it  Comm. Pure Appl. Math.} {\bf 58} (2005), 1375--1420.


\bibitem {DY2} X.T. Duong and L.X. Yan, Duality of Hardy and BMO spaces
associated with operators with heat kernel bounds. {\it J. Amer. Math. Soc.}
{\bf 18}(2005), 943--973.

\bibitem{DXY}X. Duong, J. Xiao and L. Yan, Old and new Morrey spaces with heat kernel bounds.
\textit{J. Fourier Anal. Appl.} \textbf{13} (2007), 87--111.



 \bibitem{DZ1999} J. Dziuba\'nski and J. Zienkiewicz, {Hardy space $H^1$ associated to Schr\"odinger operator with
 potential satisfying reverse H\"older inequality}.
 \textit{Rev. Mat. Iberoamericana} \textbf{15} (1999), 279--296.

\bibitem{DZ2002} J. Dziuba\'nski and J. Zienkiewicz,
{$H^p$ spaces for Schr\"odinger operators}, in: Fourier Analysis and Related Topics \textbf{56}, Banach Center Publ., Inst. Math., Polish Acad. Sci.,
Warszawa, 2002, 45--53.

\bibitem{DGMTZ} J. Dziuba\'nski, G. Garrig\'os, T. Mart\'inez, J. L. Torrea and J. Zienkiewicz,
{$BMO$ spaces related to Schr\"odinger operators with potentials satisfying a reverse H\"older inequality}.
\textit{Math. Z.}
\textbf{249} (2005), 329--356.

\bibitem{FJN}  E. Fabes, R. Johnson and U. Neri,
Spaces of harmonic functions representable by Poisson integrals of functions in BMO and $L_{p, \lambda}$.
\textit{ Indiana Univ. Math. J.}
\textbf{ 25} (1976), 159--170.


\bibitem{FN1} E. Fabes and U. Neri,
 Characterization of temperatures with initial data in BMO. \textit{ Duke Math. J. }
\textbf{42} (1975),  725-734.

\bibitem{FN} E. Fabes and U. Neri,   Dirichlet problem in Lipschitz domains with BMO data.
\textit{Proc. Amer. Math. Soc.}
\textbf{78} (1980), 33--39.



\bibitem{FS} C. Fefferman and E.M. Stein,   $H^p$ spaces of
 several variables. {\it Acta
Math.} {\bf 129} (1972), 137--195.


\bibitem{Ge} F.W. Gehring, The $L^p$-integrability of the partial derivatives of a quasiconformal mapping.
{\it Acta Math.},  {\bf 130} (1973), 265--277.



\bibitem{Gu} C.E. Guitierrez, Harnack's inequality for degenerate Schr\"odinger  operators.
{\it Trans. Amer. Math. Soc.} {\bf 312} (1989), 403--419.


\bibitem  {HLMMY} S. Hofmann, G.Z. Lu, D. Mitrea, M. Mitrea and L.X. Yan,
 Hardy spaces associated to non-negative
self-adjoint  operators satisfying Davies-Gaffney estimates.
{\it Memoirs of the Amer. Math. Soc.}, {\bf 214} (2011), no. 1007.


\bibitem{HMM} S. Hofmann, S. Mayboroda and M. Mourgoglou,  $L^p$ and endpoint solvability results
for divergence form  elliptic equations with complex $L^{\infty}$ coefficients, preprint.






\bibitem{MSTZ} T.  Ma, P.  Stinga, J.  Torrea and C. Zhang,
{Regularity properties of Schr\"odinger operators},
\textit{J. Math. Anal. Appl.}
\textbf{388} (2012), 817--837.


\bibitem{M} B. Muckenhoupt, Poisson integrals for Hermite and Laguerre expansions, {\it Trans.
Amer. Math. Soc.} {\bf 139} (1969), 231-242.


\bibitem{Shen} Z. Shen,
{$L^p$ estimates for Schr\"odinger operators with certain potentials}.
\textit{Ann. Inst. Fourier (Grenoble)}
\textbf{45} (1995), 513--546.


\bibitem{Shen1999} Z. Shen,   On fundamental solution of generalized Schr\"odinger
operators.  {\it J. Funct. Anal.}   {\bf 167}
(1999),  521--564.


\bibitem{St1970} E.M. Stein,
\textit{Topics in Harmonic Analysis Related to the Littlewood-Paley Theory},
{Annals of Mathematics Studies} \textbf{63},
Princeton Univ. Press,
Princeton, NJ, 1970.




\bibitem{SW} E. M. Stein and G. Weiss,
\textit{Introduction to Fourier Analysis on Euclidean spaces},
Princeton Univ. Press,
Princeton, NJ, 1970.

\bibitem{ST} K. Stempak and J.L. Torrea, Poisson integrals and Riesz transforms for Hermite function expansions with weights, {\it J. Funct.
Anal.} {\bf 202} (2003), 443-472.

\end{thebibliography}
\end{document}